\documentclass{article}
\usepackage{authblk}
\usepackage{a4wide}
\usepackage{amsthm}
\usepackage{graphicx}
\usepackage{amssymb}
\usepackage{amsmath}
\usepackage{ascmac}
\usepackage{setspace}
\usepackage{float}
\usepackage[dvips,usenames]{color}
\usepackage{colortbl}
\usepackage{algorithm}
\usepackage{algorithmic}
\usepackage{setspace}
\usepackage{comment}
\newtheorem{Theorem}[equation]{Theorem}
\newtheorem{Corollary}[equation]{Corollary}
\newtheorem{Lemma}[equation]{Lemma}

\theoremstyle{definition}
\newtheorem{Definition}[equation]{Definition}

\theoremstyle{remark}
\newtheorem{Remark}[equation]{Remark}
\numberwithin{equation}{section}

\DeclareMathOperator{\ev}{ev}
\DeclareMathOperator{\id}{id}

\DeclareMathOperator{\ad}{ad}

\DeclareMathOperator{\row}{row}
\DeclareMathOperator{\col}{col}

\newcommand{\ve}{\varepsilon}

\allowdisplaybreaks

\begin{document}
\title{The coproduct for the affine Yangian and the parabolic induction for non-rectangular $W$-algebras}
\author{Mamoru Ueda\thanks{mueda@ms.u-tokyo.ac.jp}}
\affil{Graduate School of Mathematical Sciences, the University of Tokyo, 3-8-1 Komaba Meguro-ku Tokyo 153-8914, Japan}
\date{}
\maketitle
\begin{abstract}
By using the coproduct and evaluation map for the affine Yangian and the Miura map for non-rectangular $W$-algebras, we construct a homomorphism from the affine Yangian associated with $\widehat{\mathfrak{sl}}(n)$ to the universal enveloping algebra of a non-rectangular $W$-algebra of type $A$, which is an affine analogue of the one given in De Sole-Kac-Valeri \cite{DKV}. As a consequence, we find that the coproduct for the affine Yangian is compatible with the parabolic induction for non-rectangular $W$-algebras via this homomorphism. It is expected that this homomorphism will be helpful for the resolution of the Crutzig-Diaconescu-Ma's conjecture. 
\end{abstract}
\textbf{keyword}: Yangian, evaluation map, $W$-algebra, coproduct

\section{Introduction}

The Yangian $Y_\hbar(\mathfrak{g})$ associated with a finite dimensional simple Lie algebra $\mathfrak{g}$ was introduced by Drinfeld (\cite{D1}, \cite{D2}). The Yangian $Y_\hbar(\mathfrak{g})$ is a quantum group which is a deformation of the current algebra $\mathfrak{g}\otimes\mathbb{C}[z]$. The affine Yangian associated with $\widehat{\mathfrak{sl}}(n)$ was first introduced by Guay (\cite{Gu2}, \cite{Gu1}).
The affine Yangian $Y_{\hbar,\ve}(\widehat{\mathfrak{sl}}(n))$ is a 2-parameter Yangian and is the deformation of the universal enveloping algebra of the central extension of $\mathfrak{sl}(n)[u^{\pm1},v]$. Guay-Nakajima-Wendlandt \cite{GNW} gave the coproduct for the Yangian associated with the Kac-Moody Lie algebra of affine type. The coproduct for the affine Yangian $Y_{\hbar,\ve}(\widehat{\mathfrak{sl}}(n))$ is a homomorphism satisfying the coassociativity:
\begin{equation*}
\Delta\colon Y_{\hbar,\ve}(\widehat{\mathfrak{sl}}(n))\to Y_{\hbar,\ve}(\widehat{\mathfrak{sl}}(n))\widehat{\otimes} Y_{\hbar,\ve}(\widehat{\mathfrak{sl}}(n)),
\end{equation*}
where $Y_{\hbar,\ve}(\widehat{\mathfrak{sl}}(n))\widehat{\otimes} Y_{\hbar,\ve}(\widehat{\mathfrak{sl}}(n))$ is the standard degreewise completion of $Y_{\hbar,\ve}(\widehat{\mathfrak{sl}}(n)){\otimes} Y_{\hbar,\ve}(\widehat{\mathfrak{sl}}(n))$. In \cite{U8} and \cite{U10}, we also constructed two kinds of a homomorphism:
\begin{align*}
\Psi_{n,m+n}&\colon Y_{\hbar,\ve}(\widehat{\mathfrak{sl}}(n))\to Y_{\hbar,\ve}(\widehat{\mathfrak{sl}}(m+n)),\\
\Psi^{m,m+n}&\colon Y_{\hbar,\ve+n\hbar}(\widehat{\mathfrak{sl}}(m))\to Y_{\hbar,\ve}(\widehat{\mathfrak{sl}}(m+n)).
\end{align*}
The images of $\Psi_{n,m+n}$ and $\Psi^{m,m+n}$ are commutative with each other and the image of $\Psi_{n,m+n}\otimes\Psi^{m,m+n}$ is the affine version of the Levi sualbgera of the finite Yangian of type $A$ given in \cite{BK0}.

Recently, the relationships between affine Yangians and $W$-algebras have been studied. The $W$-algebra $\mathcal{W}^k(\mathfrak{g},f)$ is a vertex algebra associated with a finite dimensional reductive Lie algebra $\mathfrak{g}$, a nilpotent element $f\in\mathfrak{g}$ and a complex number $k$. In \cite{U4}, we  constructed a surjective homomorphism from the affine Yangian to the universal enveloping algebra of a rectangular $W$-algebra of type $A$, which is a $W$-algebra associated with $\mathfrak{gl}(ln)$ and a nilpotent element of type $(l^n)$. Kodera and the author \cite{KU} gave another proof of the construction of this homomorphism by using the coproduct and evaluation map for the affine Yangian.
The evaluation map for the affine Yangian is a non-trivial homomorphism from the affine Yangian $Y_{\hbar,\ve}(\widehat{\mathfrak{sl}}(n))$ to the standard degreewise completion of the universal enveloping algebra of $\widehat{\mathfrak{gl}}(n)$:
\begin{equation*}
\ev^n\colon Y_{\hbar,\ve}(\widehat{\mathfrak{sl}}(n))\to \mathcal{U}(\widehat{\mathfrak{gl}}(n)),
\end{equation*}
where $\mathcal{U}(\widehat{\mathfrak{gl}}(n))$ is the standard degreewise completion of the universal enveloping algebra of $\widehat{\mathfrak{gl}}(n)$. 

In \cite{U7}, we gave a homomorphism from the affine Yangian $Y_{\hbar,\ve}(\widehat{\mathfrak{sl}}(n))$ to the universal enveloping algebra of a non-rectangular $W$-algebra of type $A$. 
However, the proof of \cite{U7} is so complicated that we cannot understand the meaning of the construction in \cite{U7}.
In this article, by the similar way to \cite{KU}, we give the another proof to \cite{U7}. Let us fix a positive integer and its partition:
\begin{gather*}
N=q_1+q_2+\cdots+q_l,\qquad q_1\leq q_2\leq\cdots\leq q_v,\ q_{v+1}\geq q_{v+2}\geq \cdots\geq q_l.
\end{gather*}
Similarly to \cite{BK}, we correspond a nilpotent element of $\mathfrak{gl}(N)$ to a diagram.
We take a nilpotent element $f\in\mathfrak{gl}(N)$ whose column heights of the diagram corresponding to $f$ is $(q_1,\cdots,q_l)$. By using the Miura map (Kac-Wakimoto \cite{KW1}), we obtain an embedding
\begin{equation*}
\mathcal{W}^k(\mathfrak{gl}(N),f)\hookrightarrow\bigotimes_{1\leq i\leq l}V^{\kappa_i}(\mathfrak{gl}(q_i)),
\end{equation*}
where $V^{\kappa_i}(\mathfrak{gl}(q_i))$ is the universal affine vertex algebra associated with $\mathfrak{gl}(q_i)$ and its inner product $\kappa_i$. The Miura map induces the injective map
\begin{equation*}
\widetilde{\mu}\colon \mathcal{U}(\mathcal{W}^k(\mathfrak{gl}(N),f))\hookrightarrow\widehat{\bigotimes}_{1\leq i\leq l}U(\widehat{\mathfrak{gl}}(q_i)),
\end{equation*}
where $\widehat{\bigotimes}_{1\leq i\leq l}U(\widehat{\mathfrak{gl}}(q_i))$ is the standard degreewise completion of $\bigotimes_{1\leq i\leq l}U(\widehat{\mathfrak{gl}}(q_i))$.

Let us set $q_{Min}$ as the minimum number of $\{q_i\}_{1\leq i\leq l}$.
By using $\Psi^{m,m+n}$, we can construct a homomorphism
\begin{equation*}
\Delta^l\colon Y_{\hbar,\ve}(\widehat{\mathfrak{sl}}(q_{Min}))\to \bigotimes_{1\leq i\leq l}Y_{\hbar,\ve-(q_i-q_{Min})\hbar}(\widehat{\mathfrak{sl}}(q_i)).
\end{equation*}
In the case that $q_1\geq q_2\geq\dots\geq q_l$, we obtain
\begin{gather*}
\Delta^l=\prod_{1\leq i\leq l-1}((\Psi^{q_{i+1},q_i}\otimes\id)\circ\Delta)\otimes\id^{\otimes l-i-1}).
\end{gather*}
\begin{Theorem}\label{A}
There exists a homomorphism
\begin{equation*}
\Phi\colon Y_{\hbar,\ve}(\widehat{\mathfrak{sl}}(q_{Min}))\to \mathcal{U}(\mathcal{W}^k(\mathfrak{gl}(N),f))
\end{equation*}
uniquely determined by
\begin{equation*}
\widetilde{\mu}\circ\Phi=\bigotimes_{1\leq i\leq l}\ev^{q_i}\circ\Delta^l.
\end{equation*}
\end{Theorem}
In \cite{U7}, we only consider the case that $q_1\geq q_2\geq\dots\geq q_l$. Theorem~\ref{A} can be applied to more general cases. One of the application of Theorem~\ref{A} is the compatibility with the coproduct for the affine Yangian and the parabolic induction for a $W$-algebra.

Let us consider two $W$-algebras associated with $(\mathfrak{gl}(N_1),f_1)$ and $(\mathfrak{gl}(N_2),f_2)$, where 
\begin{equation*}
N_1=q_1+\cdots+q_w,\ N_2=q_{w+1}+\cdots+q_l
\end{equation*}
and nilpotent elements $f_1$ and $f_2$ are defined in the same way as $f$. 
In \cite{Ge}, Genra constructed a homomorphism called the parabolic induction:
\begin{equation*}
\Delta_W\colon \mathcal{W}^{k}(\mathfrak{gl}(N),f)\to\mathcal{W}^{k+N_2}(\mathfrak{gl}(N_1),f_1)\otimes \mathcal{W}^{k+N_1}(\mathfrak{gl}(N_2),f_2).
\end{equation*}
Genra \cite{Ge} conjectured that the parabolic induction for a $W$-algebra is compatible with the coproduct for the affine Yangian.

Theorem~\ref{A} gives a relationship between the coproduct for the affine Yangian and the parabolic induction for a non-rectangular $W$-algebra. We extend the affine Yangian to two new associative algebras $Y^{p,L}_{\hbar,\ve}(\widehat{\mathfrak{sl}}(n))$ and $Y^{p,R}_{\hbar,\ve}(\widehat{\mathfrak{sl}}(n))$ for $p\geq n$. These new algebras have the following features. We can extend the coproduct $\Delta$ for the affine Yangian to
\begin{align*}
\Delta^{n}\colon Y_{\hbar,\ve}(\widehat{\mathfrak{sl}}(n))\to Y^{p,L}_{\hbar,\ve}(\widehat{\mathfrak{sl}}(n))\widehat{\otimes}Y^{p,R}_{\hbar,\ve}(\widehat{\mathfrak{sl}}(n)).
\end{align*}
For $n\leq q\leq p$, we can extend $\Psi^{n,q}$ to
\begin{align*}
\widetilde{\Psi}^{n,q,L}&\colon Y^{p,L}_{\hbar,\ve+m\hbar}(\widehat{\mathfrak{sl}}(n))\to Y^{p,L}_{\hbar,\ve}(\widehat{\mathfrak{sl}}(q)),\\
\widetilde{\Psi}^{n,q,R}&\colon Y^{p,R}_{\hbar,\ve+m\hbar}(\widehat{\mathfrak{sl}}(n))\to Y^{p,R}_{\hbar,\ve}(\widehat{\mathfrak{sl}}(q)).
\end{align*}
Then, by Theorem~\ref{A}, we obtain two homomorphisms:
\begin{gather*}
\Phi_1\colon Y^{\min\{q_w,q_{w+1}\},L}_{\hbar,\ve-(\min\{q_1,q_w\}-q_{Min})\hbar}(\widehat{\mathfrak{sl}}(\min\{q_1,q_w\})))\to \mathcal{U}(\mathcal{W}^{k+N_2}(\mathfrak{gl}(N_1),f_1)),\\
\Phi_2\colon Y^{\min\{q_w,q_{w+1}\},R}_{\hbar,\ve-(\min\{q_{w+1},q_l\}-q_{Min})\hbar}(\widehat{\mathfrak{sl}}(\min\{q_{w+1},q_l\}))\to \mathcal{U}(\mathcal{W}^{k+N_1}(\mathfrak{gl}(N_2),f_2)).
\end{gather*}
We obtain the following relations for some $a,b\in\mathbb{C}$:
\begin{gather*}
(\Phi_1\circ\tau_{a}\circ\widetilde{\Psi}^{q_l,\min\{q_1,q_w\},L})\otimes\Phi_2\circ\Delta^{n}=\Delta_W\circ\Phi\text{ if }q_1\geq q_l,\\
(\Phi_1\circ\tau_{b})\otimes(\Phi_2\circ\Psi^{q_1,\min\{q_l,q_{w+1}\},R})\circ\Delta^{n}=\Delta_W\circ\Phi\text{ if }q_1<q_l.
\end{gather*}
where $\tau_x$ is a shift operator of the affine Yangian. This means that the parabolic induction $\Delta_W$ is compatible with the coproduct $\Delta$ via the homomorphism $\Phi$.

The motivation of this theorem is the generalization of the AGT conjecture.
The AGT conjecture suggests that there exists a representation of the principal $W$-algebra of type $A$ on the equivariant homology space of the moduli space of $U(r)$-instantons. Schiffmann and Vasserot \cite{SV} gave this representation by using an action of the Yangian associated with $\widehat{\mathfrak{gl}}(1)$ on this equivariant homology space. It is conjectured by Crutzig-Diaconescu-Ma \cite{CE} that an action of an iterated $W$-algebra of type $A$ on the equivariant homology space of the affine Laumon space will be given through an action of an affine shifted Yangian constructed in \cite {FT}.
For the resoulution of the Crutzig-Diaconescu-Ma's conjecture, we need to construct a homomorphism from the shifted affine Yangian to the universal enveloping algebra of an iterated $W$-algebra. However, the shifted affine Yangian is so complicated that we cannot directly construct this homomorphism.
 
In finite setting, Brundan-Kleshchev \cite{BK} gave a surjective homomorphism from a shifted Yangian, which is a subalgebra of the finite Yangian associated with $\mathfrak{gl}(n)$, to a finite $W$-algebra (\cite{Pr}) of type $A$ for its general nilpotent element. A finite $W$-algebra $\mathcal{W}^{\text{fin}}(\mathfrak{g},f)$ is an associative algebra associated with a reductive Lie algebra $\mathfrak{g}$ and its nilpotent element $f$ and is a finite analogue of a $W$-algebra $\mathcal{W}^k(\mathfrak{g},f)$ (\cite{DSK1} and \cite{A1}). In \cite{DKV}, De Sole, Kac and Valeri constructed a homomorphism from the finite Yangian of type $A$ to the finite $W$-algebras of type $A$ by using the Lax operator, which is a restriction of the homomorphism given by Brundan-Kleshchev in \cite{BK}. Actually, the homomorphism $\Phi$ is the affine analogue of the De Sole-Kac-Valeri's homomorphism (see \cite{U9}). Thus, we expect that we can extend $\Phi$ to the shifted affine Yangian and solve the Crutzig-Diaconescu-Ma's conjecture.
\section{Affine Yangian}
Let us recall the definition of the affine Yangian of type $A$ (Definition~3.2 in \cite{Gu2} and Definition~2.3 in \cite{Gu1}). Hereafter, we sometimes identify $\{0,1,2,\cdots,n-1\}$ with $\mathbb{Z}/n\mathbb{Z}$. we also set $\{X,Y\}=XY+YX$ and
\begin{equation*}
a_{i,j} =\begin{cases}
2&\text{if } i=j, \\
-1&\text{if }j=i\pm 1,\\
0&\text{otherwise}
	\end{cases}
\end{equation*}
for $i,j\in\mathbb{Z}/n\mathbb{Z}$.
\begin{Definition}\label{Prop32}
Suppose that $n\geq3$. The affine Yangian $Y_{\hbar,\ve}(\widehat{\mathfrak{sl}}(n))$ is the associative algebra  generated by $X_{i,r}^{+}, X_{i,r}^{-}, H_{i,r}$ $(i \in \{0,1,\cdots, n-1\}, r = 0,1)$ subject to the following defining relations:
\begin{gather}
[H_{i,r}, H_{j,s}] = 0,\label{Eq2.1}\\
[X_{i,0}^{+}, X_{j,0}^{-}] = \delta_{i,j} H_{i, 0},\label{Eq2.2}\\
[X_{i,1}^{+}, X_{j,0}^{-}] = \delta_{i,j} H_{i, 1} = [X_{i,0}^{+}, X_{j,1}^{-}],\label{Eq2.3}\\
[H_{i,0}, X_{j,r}^{\pm}] = \pm a_{i,j} X_{j,r}^{\pm},\label{Eq2.4}\\
[\tilde{H}_{i,1}, X_{j,0}^{\pm}] = \pm a_{i,j}\left(X_{j,1}^{\pm}\right),\text{ if }(i,j)\neq(0,n-1),(n-1,0),\label{Eq2.5}\\
[\tilde{H}_{0,1}, X_{n-1,0}^{\pm}] = \mp \left(X_{n-1,1}^{\pm}+(\ve+\dfrac{n}{2}\hbar) X_{n-1, 0}^{\pm}\right),\label{Eq2.6}\\
[\tilde{H}_{n-1,1}, X_{0,0}^{\pm}] = \mp \left(X_{0,1}^{\pm}-(\ve+\dfrac{n}{2}\hbar) X_{0, 0}^{\pm}\right),\label{Eq2.7}\\
[X_{i, 1}^{\pm}, X_{j, 0}^{\pm}] - [X_{i, 0}^{\pm}, X_{j, 1}^{\pm}] = \pm a_{ij}\dfrac{\hbar}{2} \{X_{i, 0}^{\pm}, X_{j, 0}^{\pm}\}\text{ if }(i,j)\neq(0,n-1),(n-1,0),\label{Eq2.8}\\
[X_{0, 1}^{\pm}, X_{n-1, 0}^{\pm}] - [X_{0, 0}^{\pm}, X_{n-1, 1}^{\pm}]= \mp\dfrac{\hbar}{2} \{X_{0, 0}^{\pm}, X_{n-1, 0}^{\pm}\} + (\ve+\dfrac{n}{2}\hbar) [X_{0, 0}^{\pm}, X_{n-1, 0}^{\pm}],\label{Eq2.9}\\
(\ad X_{i,0}^{\pm})^{1+|a_{i,j}|} (X_{j,0}^{\pm})= 0 \ \text{ if }i \neq j, \label{Eq2.10}
\end{gather}
where we set $\widetilde{H}_{i,1}=H_{i,1}-\dfrac{\hbar}{2}H_{i,0}^2$.
\end{Definition}
By \eqref{Eq2.2}, \eqref{Eq2.3}, \eqref{Eq2.5} and \eqref{Eq2.7}, we find that the affine Yangian $Y_{\hbar,\ve}(\widehat{\mathfrak{sl}}(n))$ is generated by $X^\pm_{i,0}$ for $0\leq i\leq n-1$ and $H_{j,1}$ for $1\leq j\leq n-1$. 

We set a Lie algebra 
\begin{equation*}
\widehat{\mathfrak{gl}}(n)=\mathfrak{gl}(n)\otimes\mathbb{C}[t^{\pm1}]\oplus\mathbb{C}\tilde{c}\oplus\mathbb{C}z
\end{equation*}
whose commutator relations are given by
\begin{gather*}
\text{$z$ and $\tilde{c}$ are central elements of }\widehat{\mathfrak{gl}}(n),\\
\begin{align*}
[E_{i,j}\otimes t^u, E_{p,q}\otimes t^v]
&=(\delta_{j,p}E_{i,q}-\delta_{i,q}E_{p,j})\otimes t^{u+v}+\delta_{u+v,0}u(\delta_{i,q}\delta_{j,p}\widetilde{c}+\delta_{i,j}\delta_{p,q}z).
\end{align*}
\end{gather*}
Here after, in order to simplify the notation, we sometimes denote $E_{i,j}\otimes t^m\in\widehat{\mathfrak{gl}}(n)$ by $E_{i,j}t^m$.
Let us take a Lie subalgebra $\widehat{\mathfrak{sl}}(n)=\mathfrak{sl}(n)\otimes\mathbb{C}[z^{\pm1}]\oplus\mathbb{C}\tilde{c}$ and Chevalley generators of $\widehat{\mathfrak{sl}}(n)$ as
\begin{gather*}
h_i=\begin{cases}
E_{i,i}-E_{i+1,i+1}&\text{ if }i\neq0,\\
E_{n,n}-E_{1,1}+\tilde{c}&\text{ if }i=0,
\end{cases}\\
x^+_i=\begin{cases}
E_{i,i+1}&\text{ if }i\neq0,\\
E_{n,1}t&\text{ if }i=0,
\end{cases}\ x^-_i=\begin{cases}
E_{i+1,i}&\text{ if }i\neq 0,\\
E_{1,n}t^{-1}&\text{ if }i=0.
\end{cases}
\end{gather*}
By the defining relations \eqref{Eq2.1}-\eqref{Eq2.10}, we can give a homomorphism from the universal enveloping algebra of $\widehat{\mathfrak{sl}}(n)$ to the affine Yangian $Y_{\hbar,\ve}(\widehat{\mathfrak{sl}}(n))$ given by $h_i\mapsto H_{i,0}$ and $x^\pm_i\mapsto X^\pm_{i,0}$. We denote the image of $x\in U(\widehat{\mathfrak{sl}}(n))$ via this homomorphism by $x$.

We take one completion of $Y_{\hbar,\ve}(\widehat{\mathfrak{sl}}(n))$. We set the degree of $Y_{\hbar,\ve}(\widehat{\mathfrak{sl}}(n))$ by
\begin{equation*}
\text{deg}(H_{i,r})=0,\ \text{deg}(X^\pm_{i,r})=\begin{cases}
\pm1&\text{ if }i=0,\\
0&\text{ if }i\neq0.
\end{cases}
\end{equation*}
This degree is compatible with the natural degree on the universal enveloping algebra of $\widehat{\mathfrak{sl}}(n)$. We denote the standard degreewise completion of $Y_{\hbar,\ve}(\widehat{\mathfrak{sl}}(n))$ by $\widetilde{Y}_{\hbar,\ve}(\widehat{\mathfrak{sl}}(n))$ in the meaning of A.2 in \cite{A1}. 
\section{The coproduct for the affine Yangian}
By using the minimalistic presentation of the Yangian, Guay-Nakajima-Wendlandt \cite{GNW} gave a coproduct for the Yangian associated with a Kac-Moody Lie algebra of the affine type.
\begin{Theorem}[Theorem~5.2 in \cite{GNW}]
There exist algebra homomorphisms
\begin{equation*}
\Delta^\pm\colon Y_{\hbar,\ve}(\widehat{\mathfrak{sl}}(n))\to Y_{\hbar,\ve}(\widehat{\mathfrak{sl}}(n))\widehat{\otimes} Y_{\hbar,\ve}(\widehat{\mathfrak{sl}}(n))
\end{equation*}
determined by
\begin{gather*}
\Delta^\pm(X^+_{j,0})=X^+_{j,0}\otimes1+1\otimes X^+_{j,0}\text{ for }0\leq j\leq n-1,\\
\Delta^\pm(X^-_{j,0})=X^-_{j,0}\otimes1+1\otimes X^-_{j,0}\text{ for }0\leq j\leq n-1,\\
\Delta^\pm(\widetilde{H}_{i,1})=\widetilde{H}_{i,1}\otimes1+1\otimes \widetilde{H}_{i,1}+A_i^\pm\text{ for }1\leq i\leq n-1,
\end{gather*}
where $Y_{\hbar,\ve}(\widehat{\mathfrak{sl}}(n))\widehat{\otimes} Y_{\hbar,\ve}(\widehat{\mathfrak{sl}}(n))$ is the standard degreewise completion of $\otimes^2Y_{\hbar,\ve}(\widehat{\mathfrak{sl}}(n))$ and
\begin{align*}
A_i^+&=-\hbar(E_{i,i}\otimes E_{i+1,i+1}+E_{i+1,i+1}\otimes E_{i,i})\\
&\quad+\hbar\displaystyle\sum_{s \geq 0}  \limits\displaystyle\sum_{u=1}^{i}\limits (-E_{u,i}t^{-s-1}\otimes E_{i,u}t^{s+1}+E_{i,u}t^{-s}\otimes E_{u,i}t^s)\\
&\quad+\hbar\displaystyle\sum_{s \geq 0} \limits\displaystyle\sum_{u=i+1}^{n}\limits (-E_{u,i}t^{-s}\otimes E_{i,u}t^{s}+E_{i,u}t^{-s-1}\otimes E_{u,i}t^{s+1})\\
&\quad-\hbar\displaystyle\sum_{s \geq 0}\limits\displaystyle\sum_{u=1}^{i}\limits (-E_{u,i+1}t^{-s-1}\otimes E_{i+1,u}t^{s+1}+E_{i+1,u}t^{-s}\otimes E_{u,i+1}t^s)\\
&\quad-\hbar\displaystyle\sum_{s \geq 0}\limits\displaystyle\sum_{u=i+1}^{n} \limits (-E_{u,i+1}t^{-s}\otimes E_{i+1,u}t^{s}+E_{i+1,u}t^{-s-1}\otimes E_{u,i+1}t^{s+1}),\\
A_i^-&=-\hbar(E_{i,i}\otimes E_{i+1,i+1}+E_{i+1,i+1}\otimes E_{i,i})\\
&\quad+\hbar\displaystyle\sum_{s \geq 0}  \limits\displaystyle\sum_{u=1}^{i}\limits (-E_{i,u}t^{s+1}\otimes E_{u,i}t^{-s-1}+E_{u,i}t^s\otimes E_{i,u}t^{-s})\\
&\quad+\hbar\displaystyle\sum_{s \geq 0} \limits\displaystyle\sum_{u=i+1}^{n}\limits (-E_{i,u}t^{s}\otimes E_{u,i}t^{-s}+E_{u,i}t^{s+1}\otimes E_{i,u}t^{-s-1})\\
&\quad-\hbar\displaystyle\sum_{s \geq 0}\limits\displaystyle\sum_{u=1}^{i}\limits (-E_{i+1,u}t^{s+1}\otimes E_{u,i+1}t^{-s-1}+E_{u,i+1}t^s\otimes E_{i+1,u}t^{-s})\\
&\quad-\hbar\displaystyle\sum_{s \geq 0}\limits\displaystyle\sum_{u=i+1}^{n} \limits (-E_{i+1,u}t^{s}\otimes E_{u,i+1}t^{-s}+E_{u,i+1}t^{s+1}\otimes E_{i+1,u}t^{-s-1}).
\end{align*}
\end{Theorem}
The homomorphism $\Delta^\pm$ is said to be the coproduct for the affine Yangian since $\Delta^\pm$ satisfies the coassociativity.

By a direct computation, we have
\begin{align*}
\Delta(\widetilde{H}_{i,1})=\begin{cases}
\widetilde{H}_{i,1}\otimes 1+1\otimes\widetilde{H}_{i,1}+\widetilde{A}^\pm_{i}-\widetilde{A}^\pm_{i+1}\text{ if }i\neq0,\\
\widetilde{H}_{0,1}\otimes 1+1\otimes\widetilde{H}_{0,1}+\widetilde{A}^\pm_{n}-\widetilde{A}^\pm_{1}\text{ if }i=0,
\end{cases}
\end{align*}
where
\begin{align*}
\widetilde{A}_i^+&=\hbar\displaystyle\sum_{s \geq 0}  \limits\displaystyle\sum_{u=1}^{i-1}\limits (-E_{u,i}t^{-s-1}\otimes E_{i,u}t^{s+1}+E_{i,u}t^{-s}\otimes E_{u,i}t^s)\\
&\quad+\hbar\displaystyle\sum_{s \geq 0} \limits\displaystyle\sum_{u=i+1}^{n}\limits (-E_{u,i}t^{-s}\otimes E_{i,u}t^{s}+E_{i,u}t^{-s-1}\otimes E_{u,i}t^{s+1}),\\
\widetilde{A}_i^-&=\hbar\displaystyle\sum_{s \geq 0}  \limits\displaystyle\sum_{u=1}^{i-1}\limits (-E_{i,u}t^{s+1}\otimes E_{u,i}t^{-s-1}+E_{u,i}t^s\otimes E_{i,u}t^{-s})\\
&\quad+\hbar\displaystyle\sum_{s \geq 0} \limits\displaystyle\sum_{u=i+1}^{n}\limits (-E_{i,u}t^{s}\otimes E_{u,i}t^{-s}+E_{u,i}t^{s+1}\otimes E_{i,u}t^{-s-1}).
\end{align*}
\section{The evaluation map for the affine Yangian}
The evaluation map for the affine Yangian is a non-trivial homomorphism from the affine Yangian $Y_{\hbar,\ve}(\widehat{\mathfrak{sl}}(n))$ to the completion of the universal enveloping algebra of the affinization of $\mathfrak{gl}(n)$. 
We take the grading of $U(\widehat{\mathfrak{gl}}(n))/U(\widehat{\mathfrak{gl}}(n))(z-1)$ as $\text{deg}(Xt^s)=s$ and $\text{deg}(\tilde{c})=0$. We denote the degreewise completion of $U(\widehat{\mathfrak{gl}}(n))/U(\widehat{\mathfrak{gl}}(n))(z-1)$ by $\mathcal{U}(\widehat{\mathfrak{gl}}(n))$.
\begin{Theorem}[Theorem 3.8 in \cite{K1} and Theorem 4.18 in \cite{K2}]\label{thm:main}
Suppose that $\hbar\neq0$ and $\tilde{c} =\dfrac{\ve}{\hbar}$.
For $a\in\mathbb{C}$, there exists an algebra homomorphism 
\begin{equation*}
\ev_{\hbar,\ve}^{n,a} \colon Y_{\hbar,\ve}(\widehat{\mathfrak{sl}}(n)) \to \mathcal{U}(\widehat{\mathfrak{gl}}(n))
\end{equation*}
uniquely determined by 
\begin{gather*}
	\ev_{\hbar,\ve}^{n,a}(X_{i,0}^{+}) = \begin{cases}
E_{n,1}t&\text{ if }i=0,\\
E_{i,i+1}&\text{ if }1\leq i\leq n-1,
\end{cases} \ev_{\hbar,\ve}^{n,a}(X_{i,0}^{-}) = \begin{cases}
E_{1,n}t^{-1}&\text{ if }i=0,\\
E_{i+1,i}&\text{ if }1\leq i\leq n-1,
\end{cases}
\end{gather*}
and
\begin{align*}
\ev_{\hbar,\ve}^{n,a}(H_{i,1}) &=(a-\dfrac{i}{2}\hbar) \ev_{\hbar,\ve}^{n,a}(H_{i,0}) -\hbar E_{i,i}E_{i+1,i+1} \\
&\quad+ \hbar \displaystyle\sum_{s \geq 0}  \limits\displaystyle\sum_{k=1}^{i}\limits  E_{i,k}t^{-s}E_{k,i}t^s+\hbar \displaystyle\sum_{s \geq 0} \limits\displaystyle\sum_{k=i+1}^{n}\limits  E_{i,k}t^{-s-1}E_{k,i}t^{s+1}\\
&\quad-\hbar\displaystyle\sum_{s \geq 0}\limits\displaystyle\sum_{k=1}^{i}\limits E_{i+1,k}t^{-s} E_{k,i+1}t^{s}-\hbar\displaystyle\sum_{s \geq 0}\limits\displaystyle\sum_{k=i+1}^{n} \limits E_{i+1,k}t^{-s-1} E_{k,i+1}t^{s+1}
\end{align*}
for $i\neq0$.
\end{Theorem}
\begin{Remark}
The universal enveloping algebra of the universal affine vertex algebra associated with $\mathfrak{gl}(n)$ coincides with the stadard degreewise completion of the universal enveloping algebra of $\widehat{\mathfrak{gl}}(n)$. Also, the universal affine vertex algebra is a $W$-algebra associated with a nilpotent element $0$.
Thus, the evaluation map is the easiest example of a homomorphism from the affine Yangian to the universal enveloping algebra of a $W$-algebra.
\end{Remark}
\section{A homomorphism from the affine Yangian $Y_{\hbar,\ve+m\hbar}(\widehat{\mathfrak{sl}}(n))$ to the affine Yangian $Y_{\hbar,\ve}(\widehat{\mathfrak{sl}}(n+m))$}
In \cite{U8}, we gave a homomorphism from the affine Yangian $Y_{\hbar,\ve}(\widehat{\mathfrak{sl}}(n))$ to the affine Yangian $Y_{\hbar,\ve}(\widehat{\mathfrak{sl}}(n+1))$. In \cite{U10}. we gave another homomorphism from the affine Yangian $Y_{\hbar,\ve+\hbar}(\widehat{\mathfrak{sl}}(n))$ to the affine Yangian $Y_{\hbar,\ve}(\widehat{\mathfrak{sl}}(n+1))$. In this article, we only use the latter homomorphism.
\begin{Theorem}[Theorem 3.1 in \cite{U10}]\label{e}
For $m\geq 3$ and $n\geq 1$, there exists a homomorphism
    \begin{gather*}
\Psi^{m,m+n}\colon Y_{\hbar,\ve+n\hbar}(\widehat{\mathfrak{sl}}(m))\to \widetilde{Y}_{\hbar,\ve}(\widehat{\mathfrak{sl}}(m+n))
\end{gather*}
determined by
\begin{gather*}
\Psi^{m,m+n}(X^+_{i,0})=\begin{cases}
E_{m+n,n+1}t&\text{ if }i=0,\\
E_{n+i,n+i+1}&\text{ if }i\neq 0,
\end{cases}\ 
\Psi^{m,m+n}(X^-_{i,0})=\begin{cases}
E_{n+1,m+n}t^{-1}&\text{ if }i=0,\\
E_{n+i+1,n+i}&\text{ if }i\neq 0,
\end{cases}
\end{gather*}
and
\begin{align*}
\Psi^{m,m+n}(H_{i,1})&= H_{i+n,1}+\hbar\displaystyle\sum_{s \geq 0}\limits\sum_{k=1}^n E_{k,n+i}t^{-s-1}E_{n+i,k}t^{s+1}\\
&\quad-\hbar\displaystyle\sum_{s \geq 0}\limits\sum_{k=1}^n E_{k,n+i+1}t^{-s-1} E_{n+i+1,k}t^{s+1}
\end{align*}
for $i\neq0$.
\end{Theorem}
In particular, we obtain
\begin{align*}
\Psi^{m,m+n}(\widetilde{H}_{i,1})&= \widetilde{H}_{i+n,1}+\hbar\displaystyle\sum_{s \geq 0}\limits\sum_{k=1}^n E_{k,n+i}t^{-s-1}E_{n+i,k}t^{s+1}\\
&\quad-\hbar\displaystyle\sum_{s \geq 0}\limits\sum_{k=1}^n E_{k,n+i+1}t^{-s-1} E_{n+i+1,k}t^{s+1},\\
\Psi^{m,m+n}(\widetilde{H}_{0,1})&=\widetilde{H}_{0,1}+\sum_{i=1}^n\widetilde{H}_{i,1}+\dfrac{\hbar}{2}(nE_{m+n,m+n}+\sum_{u=1}^n\limits E_{u,u}+n\widetilde{c})\\
&\quad+\hbar\displaystyle\sum_{s \geq 0}\limits\sum_{k=1}^n E_{k,n+m}t^{-s-1}E_{n+m,k}t^{s+1}\\
&\quad-\hbar\displaystyle\sum_{s \geq 0}\limits\sum_{k=1}^n E_{k,m+1}t^{-s-1} E_{m+1,k}t^{s+1}.
\end{align*}
We will use another edge contraction
\section{$W$-algebras of type $A$}
We fix some notations for vertex algebras. For a vertex algebra $V$, we denote the generating field associated with $v\in V$ by $v(z)=\displaystyle\sum_{n\in\mathbb{Z}}\limits v_{(n)}z^{-n-1}$. We also denote the OPE of $V$ by
\begin{equation*}
u(z)v(w)\sim\displaystyle\sum_{s\geq0}\limits \dfrac{(u_{(s)}v)(w)}{(z-w)^{s+1}}
\end{equation*}
for all $u, v\in V$. We denote the vacuum vector (resp.\ the translation operator) by $|0\rangle$ (resp.\ $\partial$).
We denote the universal affine vertex algebra associated with a finite dimensional Lie algebra $\mathfrak{g}$ and its inner product $\kappa$ by $V^\kappa(\mathfrak{g})$. By the PBW theorem, we can identify $V^\kappa(\mathfrak{g})$ with $U(t^{-1}\mathfrak{g}[t^{-1}])$. In order to simplify the notation, here after, we denote the generating field $(ut^{-1})(z)$ as $u(z)$. By the definition of $V^\kappa(\mathfrak{g})$, the generating fields $u(z)$ and $v(z)$ satisfy the OPE
\begin{gather}
u(z)v(w)\sim\dfrac{[u,v](w)}{z-w}+\dfrac{\kappa(u,v)}{(z-w)^2}\label{OPE1}
\end{gather}
for all $u,v\in\mathfrak{g}$.

We take a positive integer and its partition:
\begin{equation*}
N=\displaystyle\sum_{i=1}^lq_i,\qquad q_1\leq q_2\leq\cdots\leq q_v,\ q_{v+1}\geq q_{v+2}\geq\cdots\geq q_l.
\end{equation*}
We also set 
\begin{align*}
q_{Max}&=\text{max}(q_v,q_{v+1}),\ q_{Min}=\text{min}(q_1,q_{l}).
\end{align*}
We also fix an inner product of $\mathfrak{gl}(N)$ determined by
\begin{equation*}
(E_{i,j}|E_{p,q})=k\delta_{i,q}\delta_{p,j}+\delta_{i,j}\delta_{p,q}.
\end{equation*}
For $1\leq i\leq N$, we set $1\leq \col(i)\leq l$ and $q_{Max}-q_{\col(i)}<\row(i)\leq q_{Max}$ satisfying
\begin{gather*}
\col(i)=s\text{ if }\sum_{j=1}^{s-1}q_j<i\leq\sum_{j=1}^sq_j,\ 
\row(i)=i-\sum_{j=1}^{\col(i)-1}q_j+q_{Max}-q_{\col(i)}.
\end{gather*}
We take a nilpotent element $f$ as 
\begin{equation*}
f=\sum_{1\leq j\leq N}\limits e_{\hat{j},j},
\end{equation*}
where the integer $1\leq \hat{j}\leq N$ is determined by
\begin{equation*}
\col(\hat{j})=\col(j)+1,\ \row(\hat{j})=\row(j).
\end{equation*}
We consider the following Lie superalgebras:
\begin{gather*}
\mathfrak{b}=\bigoplus_{\substack{1\leq i,j\leq N\\\col(i)\geq\col(j)}}\limits\mathbb{C}e_{i,j},\\
\mathfrak{a}=\mathfrak{b}\oplus\displaystyle\bigoplus_{\substack{1\leq i,j\leq N\\\col(i)>\col(j)}}\limits\mathbb{C}\psi_{i,j}
\end{gather*}
whose commutator relations;
\begin{align*}
[e_{i,j},\psi_{p,q}]&=\delta_{j,p}\psi_{i,q}-\delta_{i,q}\psi_{p,j},\\
[\psi_{i,j},\psi_{p,q}]&=0,
\end{align*}
where $e_{i,j}$ is an even element and $\psi_{i,j}$ is an odd element. We set the inner product on $\mathfrak{a}$ such that
\begin{gather*}
\widetilde{\kappa}(e_{i,j},e_{p,q})=\kappa(e_{i,j},e_{p,q}),\qquad\widetilde{\kappa}(e_{i,j},\psi_{p,q})=\widetilde{\kappa}(\psi_{i,j},\psi_{p,q})=0.
\end{gather*}
By the PBW theorem, we identify $V^{\widetilde{\kappa}}(\mathfrak{a})$ with $U(t^{-1}\mathfrak{a}[t^{-1}])$. For all $u\in \mathfrak{a}$, let $u[-s]$ be $ut^{-s}$. In this section, we regard $V^{\widetilde{\kappa}}(\mathfrak{a})$ (resp.\ $V^{\widetilde{\kappa}|_{\mathfrak{b}}}(\mathfrak{b})$) as a non-associative superalgebra structure by $(-1)$-product. In particular, we obtain
\begin{equation*}
u[-w]\cdot v[-s]=(u[-w])_{(-1)}v[-s].
\end{equation*}
We sometimes omit $\cdot$ in order to simplify the notation. 

We denote the $W$-algebra associated with $\mathfrak{gl}(N)$ and $f$ by $\mathcal{W}^k(\mathfrak{gl}(N),f)$. By \cite{KW1} and \cite{KW2}, a $W$-algebra $\mathcal{W}^k(\mathfrak{gl}(N),f)$ can be realized as the vertex subalgebra of $V^\kappa(\mathfrak{b})$.
Let us define an odd differential $d_0 \colon V^{\kappa}(\mathfrak{b})\to V^{\widetilde{\kappa}}(\mathfrak{a})$ determined by
\begin{gather}
d_0(|0\rangle)=0,\\
[d_0,\partial]=0,\label{ee5800}
\end{gather}
\begin{align}
d_0(e_{i,j}[-1])
&=\sum_{\substack{\col(i)>\col(r)\geq\col(j)}}\limits e_{r,j}[-1]\psi_{i,r}[-1]-\sum_{\substack{\col(j)<\col(r)\leq\col(i)}}\limits \psi_{r,j}[-1]e_{i,r}[-1]\nonumber\\
&\quad+\delta(\col(i)>\col(j))\alpha_{\col(i)}\psi_{i,j}[-2]+\psi_{\hat{i},j}[-1]-\psi_{i,\tilde{j}}[-1].\label{ee1}
\end{align}
Here, we assumed that $e_{p,q}=0$ and $\psi_{p,q}= 0$ for out-of-range subscripts. We also supposed that $X_{\hat{i},j}=0$ and $X_{i,\tilde{j}}=0$ ($X=e,\psi$) if $\hat{i}$ and $\tilde{j}$ do not exist.
\begin{Definition}
The $W$-algebra $\mathcal{W}^k(\mathfrak{gl}(m+n),f)$ is the vertex subalgebra of $V^\kappa(\mathfrak{b})$ defined by
\begin{equation*}
\mathcal{W}^k(\mathfrak{gl}(m+n),f)=\{y\in V^\kappa(\mathfrak{b})\mid d_0(y)=0\}.
\end{equation*}
\end{Definition}
We can prove the following theorem in the same way as Theorem~4.6 in \cite{U7}.
\begin{Theorem}
We set $\gamma_a=\sum_{u=a+1}^{l}\limits \alpha_{u}$. The following elemets of $\bigotimes_{1\leq s\leq l}V^{\kappa_s}(\mathfrak{gl}(q_s))$ are contained in $\mathcal{W}^k(\mathfrak{gl}(N),f)$.
\begin{align*}
W^{(1)}_{p,q}&=\sum_{\substack{1\leq\col(i)=\col(j)\leq l,\\\row(i)=p,\row(j)=q}}e_{i,j}[-1],\\
&\qquad\qquad\text{ for }p>q_{Max}-q_{Min},q>q_{Max}-q_l\text{ or }p>q_{Max}-q_1,q>q_{Max}-q_{Min},\\
W^{(2)}_{p,q}&=\sum_{\substack{\col(i)=\col(j)+1,\\\row(i)=p,\row(j)=q}}e_{i,j}[-1]-\sum_{\substack{1\leq \col(i)=\col(j)\leq l,\\\row(i)=p,\row(j)=q}}\gamma_{r}e_{i,j}[-2]\\
&\quad+\sum_{\substack{\col(i)=\col(u)<\col(j)=\col(v)\\\row(u)=\row(v)>q_{Max}-q_{Min}\\\row(i)=p,\row(j)=q}}\limits e_{u,j}[-1]e_{i,v}[-1]\\
&\quad-\sum_{\substack{\col(i)=\col(u)\geq \col(j)=\col(v)\\q_{Max}-\text{min}(q_{r_1},q_{r_2})<\row(u)=\row(v)\leq q_{Max}-q_{Min}\\\row(i)=p,\row(j)=q}}\limits e^{(r_1)}_{u,j}[-1]e^{(r_2)}_{i,v}[-1]\text{ for }p,q>q_{Max}-q_{Min}.
\end{align*}
\end{Theorem}
\begin{proof}
We only show that $W^{(2)}_{p,q}$ is contained in $\mathcal{W}^k(\mathfrak{gl}(N),f)$.
By \eqref{ee1}, we have
\begin{align}
[d_0,e_{i,j}[-1]]
&=\psi_{\widehat{i},j}[-1]-\psi_{i,\widetilde{j}}[-1]\label{ee307}
\end{align}
if $\col(i)=\col(j)$ and 
\begin{align}
&\quad[d_0,e_{i,j}[-1]]\nonumber\\
&=\sum_{\substack{\col(r)=\col(j)}}\limits e_{r,j}[-1]\psi_{i,r}[-1]-\sum_{\substack{\col(r)=\col(i)}}\limits \psi_{r,j}[-1]e_{i,r}[-1]\nonumber\\
&\quad+\alpha_{\col(i)}\psi_{i,j}[-2]+\psi_{\hat{i},j}[-1]-\psi_{i,\tilde{j}}[-1].\label{ee2}
\end{align}
if $\col(i)=\col(j)+1$.
By the definition of $W^{(2)}_{i,j}$, we can rewrite $d_0(W^{(2)}_{p,q})$ as
\begin{align}
&\sum_{\substack{\col(i)=\col(j)+1\\\row(i)=p,\row(j)=q}}d_0(e_{i,j}[-1])-\sum_{\substack{\col(i)=\col(j)\\\row(i)=p,\row(j)=q}}\gamma_{\col(i)}d_0(e_{i,j}[-2])\nonumber\\
&\quad+\sum_{\substack{\col(u)=\col(j)<\col(i)=\col(v)\\\row(u)=\row(v)\leq q_l\\\row(i)=p,\row(j)=q}}\limits d_0(e_{u,j}[-1])e_{i,v}[-1]\nonumber\\
&\quad+\sum_{\substack{\col(u)=\col(j)<\col(i)=\col(v)\\\row(u)=\row(v)\leq q_l\\\row(i)=p,\row(j)=q}}\limits e_{u,j}[-1]d_0(e_{i,v}[-1])\nonumber\\
&\quad-\sum_{\substack{\col(u)=\col(j)\geq\col(i)=\col(v)\\\row(u)=\row(v)>q_l\\\row(i)=p,\row(j)=q}}\limits d_0(e_{u,j}[-1])e_{i,v}[-1]\nonumber\\
&\quad-\sum_{\substack{\col(u)=\col(j)\geq\col(i)=\col(v)\\\row(u)=\row(v)>q_l\\\row(i)=p,\row(j)=q}}\limits e_{u,j}[-1]d_0(e_{i,v}[-1]).\label{ee3}
\end{align}
We denote the $i$-th term of the right hand side of the equation$(\cdot)$ by $(\cdot)_i$. By \eqref{ee2}, we have
\begin{align}
\eqref{ee3}_1
&=\sum_{\substack{\col(i)=\col(j)+1\\\row(i)=p,\row(j)=q}}\sum_{\substack{\col(r)=\col(j)}}\limits e_{r,j}[-1]\psi_{i,r}[-1]-\sum_{\substack{\col(i)=\col(j)+1\\\row(i)=p,\row(j)=q}}\sum_{\substack{\col(r)=\col(i)}}\limits \psi_{r,j}[-1]e_{i,r}[-1]\nonumber\\
&\quad+\sum_{\substack{\col(i)=\col(j)+1\\\row(i)=p,\row(j)=q}}\alpha_{\col(i)}\psi_{i,j}[-2].\label{ee5.1}
\end{align}
By \eqref{ee307} and \eqref{ee5800}, we obtain
\begin{align}
\eqref{ee3}_2
&=-\sum_{\substack{\col(i)=\col(j)\\\row(i)=p,\row(j)=q}}\limits\alpha_{\col(\hat{i})}\psi_{\hat{i},j}[-2].\label{ee6}
\end{align}
By \eqref{ee307}, we obtain
\begin{align}
\eqref{ee3}_3
&=\sum_{\substack{\col(u)+1=\col(j)+1=\col(i)=\col(v)\\\row(u)=\row(v)\leq q_l\\\row(i)=p,\row(j)=q}}\limits \psi_{\hat{u},j}[-1]e_{i,v}[-1]\label{ee7},\\
\eqref{ee3}_4
&=-\sum_{\substack{\col(u)+1=\col(j)+1=\col(i)=\col(v)\\\row(u)=\row(v)\leq q_l\\\row(i)=p,\row(j)=q}}\limits e_{u,j}[-1]\psi_{i,\tilde{v}}[-1],\label{ee8}\\
\eqref{ee3}_5
&=\sum_{\substack{\col(u)=\col(j)=\col(i)+1=\col(v)+1\\q_l<\row(u)=\row(v)\leq q_{\col(j)}\\\row(i)=p,\row(j)=q}}\limits\psi_{u,\tilde{j}}[-1]e_{i,v}[-1],\label{ee9}\\
\eqref{ee3}_6
&=-\sum_{\substack{\col(u)=\col(j)=\col(i)+1=\col(v)+1\\q_l<\row(u)=\row(v)\leq q_{\col(j)}\\\row(i)=p,\row(j)=q}}\limits e_{u,j}[-1]\psi_{\hat{i},v}[-1].\label{ee10}
\end{align}
By adding \eqref{ee5.1}-\eqref{ee10}, we obtain $d_0(W^{(2)}_{i,j})=0$.
\end{proof}
We set the inner product on $\mathfrak{gl}(q_s)$ by
\begin{equation*}
\kappa_s(E_{i,j},E_{p,q})=\delta_{j,p}\delta_{i,q}\alpha_s+\delta_{i,j}\delta_{p,q},
\end{equation*}
where $\alpha_s=k+N-q_s$. Then, by Corollary 5.2 in \cite{Genra}, the $W$-algebra $\mathcal{W}^k(\mathfrak{gl}(N),f)$ can be embedded into $\bigotimes_{1\leq s\leq l}V^{\kappa_s}(\mathfrak{gl}(q_s))$. This embedding is called the Miura map. Hereafter, we denote $E_{i,j}t^{-u}\in U(t^{-1}\mathfrak{gl}(q_s)[t^{-1}])=V^\kappa(\mathfrak{gl}(q_s))$ by $e_{q_{Max}-q_s+i,q_{Max}-q_s+j}[-u]$.
\begin{Corollary}\label{Generators}
We set $\gamma_a=\sum_{u=a+1}^{l}\limits \alpha_{u}$. For positive integers $p,q>q_{Max}-q_{Min}$, we have
\begin{align*}
\mu(W^{(1)}_{p,q})&=\sum_{\substack{1\leq r\leq l}}e^{(r)}_{p,q}[-1],\\
\mu(W^{(2)}_{p,q})&=-\sum_{\substack{1\leq r\leq l}}\gamma_{r}e^{(r)}_{p,q}[-2]+\sum_{\substack{r_1<r_2\\u>q_{Max}-q_{Min}}}\limits e^{(r_1)}_{u,q}[-1]e^{(r_2)}_{p,u}[-1]\\
&\quad-\sum_{\substack{r_1\geq r_2\\q_{Max}-\text{min}(q_{r_1},q_{r_2})<u\leq q_{Max}-q_{Min}}}\limits e^{(r_1)}_{u,q}[-1]e^{(r_2)}_{p,u}[-1],
\end{align*}
where $x^{(s)}=1^{\otimes s-1}\otimes x\otimes1^{l-s}$.
\end{Corollary}
For the latter proof, we will give some OPEs of $W^{(1)}_{i,j}$ and $W^{(2)}_{p,q}$. The proof is given by a direct computation.
\begin{Lemma}\label{OPE}
The following relations hold:
\begin{align*}
(W^{(1)}_{i,j})_{(0)}W^{(2)}_{p,q}&=\delta(i\leq q_{Max}-q_{Min},j>q_{Max}-q_{Min})(W^{(1)}_{i,q})_{(-1)}W^{(1)}_{p,j}\\
&\quad-\delta(i>q_{Max}-q_{Min},j\leq q_{Max}-q_{Min})(W^{(1)}_{i,q})_{(-1)}W^{(1)}_{p,j},\\
(W^{(1)}_{i,j})_{(1)}W^{(2)}_{p,q}&=-\delta(i>q_{Max}-q_{Min},j\leq q_{Max}-q_{Min})\delta_{p,q}W^{(1)}_{i,j},\\
(W^{(1)}_{i,j})_{(r)}W^{(2)}_{p,q}&=0\text{ for }r\geq2
\end{align*}
for $i\leq q_{Max}-q_{Min},j>q_{Max}-q_{Min}$ or $i>q_{Max}-q_{Min},j\leq q_{Max}-q_{Min}$ satisfying that $i\neq q$ and $j\neq p$.
\end{Lemma}

Let us recall the definition of a universal enveloping algebra of a vertex algebra in the sense of \cite{FZ} and \cite{MNT}.
For any vertex algebra $V$, let $L(V)$ be the Borcherds Lie algebra, that is,
\begin{align}
 L(V)=V{\otimes}\mathbb{C}[t,t^{-1}]/\text{Im}(\partial\otimes\id +\id\otimes\frac{d}{d t})\label{844},
\end{align}
where the commutation relation is given by
\begin{align*}
 [ut^a,vt^b]=\sum_{r\geq 0}\begin{pmatrix} a\\r\end{pmatrix}(u_{(r)}v)t^{a+b-r}
\end{align*}
for all $u,v\in V$ and $a,b\in \mathbb{Z}$. 
\begin{Definition}[Section~6 in \cite{MNT}]\label{Defi}
We set $\mathcal{U}(V)$ as the quotient algebra of the standard degreewise completion of the universal enveloping algebra of $L(V)$ by the completion of the two-sided ideal generated by
\begin{gather}
(u_{(a)}v)t^b-\sum_{i\geq 0}
\begin{pmatrix}
 a\\i
\end{pmatrix}
(-1)^i(ut^{a-i}vt^{b+i}-(-1)^avt^{a+b-i}ut^{i}),\label{241}\\
|0\rangle t^{-1}-1.
\end{gather}
We call $\mathcal{U}(V)$ the universal enveloping algebra of $V$.
\end{Definition}
Induced by the Miura map $\mu$, we obtain the embedding
\begin{equation*}
\widetilde{\mu}\colon \mathcal{U}(\mathcal{W}^{k}(\mathfrak{gl}(N),f))\to {\widehat{\bigoplus}}_{1\leq s\leq l}U(\widehat{\mathfrak{gl}}(q_s)),
\end{equation*}
where ${\widehat{\bigotimes}}_{1\leq s\leq l}U(\widehat{\mathfrak{gl}}(q_s))$ is the standard degreewise completion of $\bigotimes_{1\leq s\leq l}U(\widehat{\mathfrak{gl}}(q_s))$. 

By the definition of $W^{(1)}_{i,j}$ and $W^{(2)}_{i,j}$, we have
\begin{align}
\widetilde{\mu}(W^{(1)}_{i,j}t^s)&=\sum_{1\leq r\leq l}\limits e^{(r)}_{i,j}t^s,\label{W1}\\
\widetilde{\mu}(W^{(2)}_{i,j}t)&=\sum_{r=1}^n\gamma_re^{(r)}_{i,j}+\sum_{s\in\mathbb{Z}}\limits\sum_{r_1<r_2}\sum_{u>q_{Max}-q_{Min}}\limits e^{(r_1)}_{u,j}t^{-s}e^{(r_2)}_{i,u}t^s\nonumber\\
&\quad-\sum_{s\geq0}\limits\sum_{r\geq0}\limits\sum_{q_{Max}-q_r\leq u \leq q_{Max}-q_{Min}}\limits (e^{(r)}_{u,j}t^{-s-1}e^{(r)}_{i,u}t^{s+1}+e^{(r)}_{i,u}t^{-s}e^{(r)}_{u,j}t^{s})\nonumber\\
&\quad-\sum_{s\in\mathbb{Z}}\limits\sum_{r_1<r_2}\sum_{q_{Max}-\text{min}(q_{r_1},q_{r_2})<u\leq q_{Max}-q_{Min}}\limits e^{(r_1)}_{i,u}t^{-s}e^{(r_2)}_{u,j}t^s.\label{W2}
\end{align}
\section{A parabolic induction for a $W$-algebra of type $A$}
Let us recall the parabolic induction for $W$-algebras of type $A$ (see Theorem 6.1 in \cite{Ge}). Let us take two positive integers
\begin{align*}
N_1=q_1+\cdots+q_w,\ N_2=q_{w+1}+\cdots+q_l.
\end{align*}
and $f_1\in\mathfrak{gl}(N_1)$ and $f_2\in\mathfrak{gl}(N_2)$ are nilpotent elements defined by the same way as $f$. Let us denote the Miura map associated with $\mathcal{W}^{k+N_2}(\mathfrak{gl}(N_1),f_1)$ and $\mathcal{W}^{k+N_1}(\mathfrak{gl}(N_2),f_2)$ by $\mu_1$ and $\mu_2$.
\begin{Theorem}[Theorem 6.1 in \cite{Ge}]
There exists a homomorphism
\begin{equation}
\Delta_W\colon\mathcal{W}^{k}(\mathfrak{gl}(N),f)\to\mathcal{W}^{k+N_2}(\mathfrak{gl}(N_1),f_1)\otimes \mathcal{W}^{k+N_1}(\mathfrak{gl}(N_2),f_2)\label{para1}
\end{equation}
satisfying that
\begin{equation}
\mu=(\mu_1\otimes\mu_2)\circ\Delta_W.\label{para}
\end{equation}
\end{Theorem}
We call this homomorphism the parabolic induction for a $W$-algebra. By \eqref{para}, we find that
\begin{align*}
\Delta_W(W^{(1)}_{i,j})&=W^{(1)}_{i,j}\otimes 1+1\otimes W^{(1)}_{i,j},\\
\Delta_W(W^{(2)}_{i,j})&=\begin{cases}
W^{(2)}_{i,j}\otimes 1+1\otimes W^{(2)}_{i,j}-\gamma_{w}\partial W^{(1)}_{i,j}\otimes 1\\
\quad-\sum_{q_{Max}-\text{min}(q_1,q_w)< u\leq q_{Max}-q_{Min}}(W^{(1)}_{u,j})_{(-1)}W^{(1)}_{i,u}\otimes 1\\
\quad-\sum_{u> q_{Max}-q_{Min}}(W^{(1)}_{u,j})\otimes W^{(1)}_{i,u}\\
\quad-\sum_{q_{Max}-\text{min}\{q_w,q_{w+1}\}<u\leq q_{Max}-q_{Min}}W^{(1)}_{i,u}\otimes(W^{(1)}_{u,j})\text{ if }q_1\geq q_l,\\
W^{(2)}_{i,j}\otimes 1+1\otimes W^{(2)}_{i,j}-\gamma_{w}\partial W^{(1)}_{i,j}\otimes 1\\
\quad-1\otimes\sum_{q_{Max}-\text{min}(q_{w+1},q_l)< u\leq q_{Max}-q_{Min}}(W^{(1)}_{u,j})_{(-1)}W^{(1)}_{i,u}\\
\quad-\sum_{u> q_{Max}-q_{Min}}(W^{(1)}_{u,j})\otimes W^{(1)}_{i,u}\\
\quad-\sum_{q_{Max}-\text{min}\{q_w,q_{w+1}\}<u\leq q_{Max}-q_{Min}}W^{(1)}_{i,u}\otimes(W^{(1)}_{u,j})\text{ if }q_1< q_l.
\end{cases}
\end{align*}
We also denote by $\Delta_W$ a homomorphism 
\begin{equation*}
\Delta_W\colon\mathcal{U}(\mathcal{W}^{k}(\mathfrak{gl}(N),f))\to\mathcal{U}(\mathcal{W}^{k+N_2}(\mathfrak{gl}(N_1),f_1))\otimes \mathcal{U}(\mathcal{W}^{k+N_1}(\mathfrak{gl}(N_2),f_2)),
\end{equation*}
which is induced by $\Delta_W$ in \eqref{para1}.
\section{A homomorphism from the affine Yangian to the universal enveloping algebra of $\mathcal{W}^{k}(\mathfrak{gl}(N),f)$}
In this section, we will give the another proof to Theorem 6.1 in \cite{U7}. The following theorem is the same as Theorem 6.1 in \cite{U7}. In \cite{U7}, we gave the proof to Theorem~\ref{Main} by a direct computation. 
In this article, we prove it by showing the following theorem. Let us take the permutation of $\sigma\in S_l$ satisfying that
\begin{equation*}
\begin{cases}
q_{\sigma(1)}\geq q_{\sigma(2)}\geq\cdots\geq q_{\sigma(l)},\\
\sigma(i)<\sigma(j)\text{ if }i<j\text{ and }q_i=q_j.
\end{cases}
\end{equation*}
Let us set a homomorphism
Let us set a homomorphism
\begin{align*}
\Delta^l&=(\prod_{s=1}^{l-1}\Delta^{q_{\sigma(s)},q_{\sigma(s+1)}}\otimes\id^{l-s-1})\colon Y_{\hbar,\ve}(\widehat{\mathfrak{sl}}(q_{Min}))\to\bigotimes_{1\leq i\leq l}Y_{\hbar,\ve-\hbar(q_{\sigma(i)-q_{Min}}}(\widehat{\mathfrak{sl}}(q_{\sigma(i)}))
\end{align*}
where we set
\begin{equation*}
\Delta^{q_{\sigma(s)},q_{\sigma(s+1)}}=\begin{cases}
(\Psi^{q_{\sigma(s+1)},q_{\sigma(s)}}\otimes1)\circ\Delta^+\text{ if }\sigma(s)<\sigma(s+1),\\
(\Psi^{q_{\sigma(s+1)},q_{\sigma(s)}}\otimes1)\circ\Delta^-\text{ if }\sigma(s)>\sigma(s+1).
\end{cases}
\end{equation*}

We also define a natural homomorphism
\begin{equation*}
\Sigma\colon \bigotimes_{1\leq i\leq l}Y_{\hbar,\ve-\hbar(q_{\sigma(i)}-q_{Min})}(\widehat{\mathfrak{sl}}(q_{\sigma(i)}))\to\bigotimes_{1\leq i\leq l} Y_{\hbar,\ve-\hbar(q_{i}-q_{Min})}(\widehat{\mathfrak{sl}}(q_{i})).
\end{equation*}
Let us set $\Phi(X^\pm_{i,0}),\Phi(H_{i,1})\in\mathcal{U}(\mathcal{W}^k(\mathfrak{gl}(N),f))$ as
\begin{gather*}
\Phi(X^+_{i,0})=\begin{cases}
W^{(1)}_{q_{Max},q_{Max}-q_{Min}+1}t&\text{ if }i=0,\\
W^{(1)}_{q_{Max}-q_{Min}+i,q_{Max}-q_{Min}+i+1}&\text{ if }i\neq0,
\end{cases}\\
\Phi(X^-_{i,0})=\begin{cases}
W^{(1)}_{q_{Max}-q_{Min}+1,q_{Max}}t^{-1}&\text{ if }i=0,\\
W^{(1)}_{q_{Max}-q_{Min}+i+1,q_{Max}-q_{Min}+i}&\text{ if }i\neq0,
\end{cases}
\end{gather*}
and
\begin{align*}
\Phi(H_{i,1})&=
-\hbar(W^{(2)}_{q_{Max}-q_{Min}+i,q_{Max}-q_{Min}+i}t-W^{(2)}_{q_{Max}-q_{Min}+i+1,q_{Max}-q_{Min}+i+1}t)\\
&\quad-\dfrac{i}{2}\hbar(W^{(1)}_{q_{Max}-q_{Min}+i,q_{Max}-q_{Min}+i}-W^{(1)}_{q_{Max}-q_{Min}+i+1,q_{Max}-q_{Min}+i+1})\\
&\quad+\hbar W^{(1)}_{q_{Max}-q_{Min}+i,q_{Max}-q_{Min}+i}W^{(1)}_{q_{Max}-q_{Min}+i+1,q_{Max}-q_{Min}+i+1}\\
&\quad+\hbar\displaystyle\sum_{s \geq 0}  \limits\displaystyle\sum_{u=1}^{i}\limits W^{(1)}_{q_{Max}-q_{Min}+i,q_{Max}-q_{Min}+u}t^{-s}W^{(1)}_{q_{Max}-q_{Min}+u,q_{Max}-q_{Min}+i}t^s\\
&\quad+\hbar\displaystyle\sum_{s \geq 0} \limits\displaystyle\sum_{u=i+1}^{n}\limits W^{(1)}_{q_{Max}-q_{Min}+i,q_{Max}-q_{Min}+u}t^{-s-1} W^{(1)}_{q_{Max}-q_{Min}+u,q_{Max}-q_{Min}+i}t^{s+1}\\
&\quad-\hbar\displaystyle\sum_{s \geq 0}\limits\displaystyle\sum_{u=1}^{i}\limits W^{(1)}_{q_{Max}-q_{Min}+i+1,q_{Max}-q_{Min}+u}t^{-s} W^{(1)}_{q_{Max}-q_{Min}+u,q_{Max}-q_{Min}+i+1}t^s\\
&\quad-\hbar\displaystyle\sum_{s \geq 0}\limits\displaystyle\sum_{u=i+1}^{n} \limits W^{(1)}_{q_{Max}-q_{Min}+i+1,q_{Max}-q_{Min}+u}t^{-s-1} W^{(1)}_{q_{Max}-q_{Min}+u,q_{Max}-q_{Min}+i+1}t^{s+1}
\end{align*}
for $i\neq0$.
\begin{Theorem}\label{hojo}
We obtain the following relation:
\begin{equation*}
\bigotimes_{1\leq s\leq l}\ev_{\hbar,\ve-(q_s-q_{Min})\hbar}^{q_s,(\gamma_s-\frac{q_s-q_{Min}}{2})\hbar}\circ\Sigma\circ\Delta^l=\widetilde{\mu}\circ\Phi.
\end{equation*}
\end{Theorem}
Since $\widetilde{\mu}$ is an injective homomorphism (\cite{Genra}) and $\ev_{\hbar,\ve+(q_s-q_{Min})\hbar}^{q_s,\gamma_s\hbar}$, $\Delta^l$ and $\widetilde{\mu}$ are homomorphisms, we obtain the following theorem.
\begin{Theorem}\label{Main}
Assume that $q_{Min}\geq 3$ and $\ve=\hbar(k+N-q_{Min})$. The definition of $\Phi$ induces a homomorphism 
\begin{equation*}
\Phi\colon Y_{\hbar,\ve}(\widehat{\mathfrak{sl}}(q_{Min}))\to \mathcal{U}(\mathcal{W}^{k}(\mathfrak{gl}(N),f))
\end{equation*} 
\end{Theorem}
\begin{Remark}
We note that the Definition~\ref{Prop32} is different from the one in \cite{U7}. The parameter $\ve$ in this article is equal to $-n\hbar-\ve$ in \cite{U7}. This is the reason that the assumption of $\ve$ is different from the one in \cite{U7}.
\end{Remark}
The rest of this section is devoted to Theorem~\ref{hojo}.
\begin{proof}[The proof of Theorem~\ref{hojo}]
In order to simplify the notation, we only show the case that $q_1\geq q_2\geq \cdots\geq q_l$. The other cases can be proven in a similar way.

It is enough to show the relations:
\begin{equation}
\bigotimes_{1\leq s\leq l}\ev_{\hbar,\ve-(q_s-q_{Min})\hbar}^{q_s,(\gamma_s-\frac{q_s-q_{Min}}{2})\hbar}\circ\Delta^l(X^\pm_{j,0})=\widetilde{\mu}\circ\Phi(X^\pm_{j,0})\label{111}
\end{equation}
for $0\leq j\leq n-1$ and
\begin{equation}
\bigotimes_{1\leq s\leq l}\ev_{\hbar,\ve-(q_s-q_{Min})\hbar}^{q_s,(\gamma_s-\frac{q_s-q_{Min}}{2})\hbar}\circ\Delta^l(H_{i,1})=\widetilde{\mu}\circ\Phi(H_{i,1})\label{112}
\end{equation}
for $1\leq i\leq n-1$. It is trivial that \eqref{111} holds. We only show the relation \eqref{112}. For $A\in\widetilde{Y}_{\hbar,\ve}(\widehat{\mathfrak{sl}}(n)),\mathcal{U}(\widehat{\mathfrak{gl}}(q_s))$, we denote $A^{(s)}=1^{\otimes s-1}\otimes A\otimes 1^{\otimes l-s}$.
By the definition of $\Psi^{n,m+n}$ and $\Delta^n$, we obtain
\begin{align*}
\Delta^l(H_{i,1})&=\sum_{1\leq s\leq l}\limits H^{(s)}_{i+q_s-q_l}+B_i+C_i,
\end{align*}
where we set
\begin{align*}
B_i&=-\hbar\sum_{r_1<r_2}\limits E^{(r_1)}_{q_{r_1}-q_l+i,q_{r_1}-q_l+i}E^{(r_2)}_{q_{r_2}-q_l+i+1,q_{r_2}-q_l+i+1}\\
&\quad-\hbar\sum_{r_1<r_2}\limits E^{(r_1)}_{q_{r_1}-q_l+i+1,q_{r_1}-q_l+i+1}E^{(r_2)}_{q_{r_2}-q_l+i,q_{r_2}-q_l+i}\\
&\quad-\hbar\sum_{r_1<r_2}\limits\displaystyle\sum_{s \geq 0}  \limits\displaystyle\sum_{u=1}^{q_{r_2}-q_l+i}\limits E^{(r_1)}_{q_{r_1}-q_{r_2}+u,q_{r_1}-q_l+i}t^{-s-1}E^{(r_2)}_{q_{r_2}-q_l+i,u}t^{s+1}\\
&\quad+\hbar\sum_{r_1<r_2}\limits\displaystyle\sum_{s \geq 0}  \limits\displaystyle\sum_{u=1}^{q_{r_2}-q_l+i}\limits E^{(r_1)}_{q_{r_1}-q_l+i,q_{r_1}-q_{r_2}+u}t^{-s}E^{(r_2)}_{u,q_{r_2}-q_l+i}t^s\\
&\quad-\hbar\sum_{r_1<r_2}\limits\displaystyle\sum_{s \geq 0} \limits\displaystyle\sum_{u=q_{r_2}-q_l+i+1}^{q_{r_2}}\limits E^{(r_1)}_{q_{r_1}-q_{r_2}+u,q_{r_1}-q_l+i}t^{-s}E^{(r_2)}_{q_{r_2}-q_l+i,u}t^{s}\\
&\quad+\hbar\sum_{r_1<r_2}\limits\displaystyle\sum_{s \geq 0} \limits\displaystyle\sum_{u=q_{r_2}-q_l+i+1}^{q_{r_2}}\limits E^{(r_1)}_{q_{r_1}-q_l+i,q_{r_1}-q_{r_2}+u}t^{-s-1}E^{(r_2)}_{u,q_{r_2}-q_l+i}t^{s+1}\\
&\quad+\hbar\sum_{r_1<r_2}\limits\displaystyle\sum_{s \geq 0}\limits\displaystyle\sum_{u=1}^{q_{r_2}-q_l+i}\limits E^{(r_1)}_{q_{r_1}-q_{r_2}+u,q_{r_1}-q_l+i+1}t^{-s-1}E^{(r_2)}_{q_{r_2}-q_l+i+1,u}t^{s+1}\\
&\quad-\hbar\sum_{r_1<r_2}\limits\displaystyle\sum_{s \geq 0}\limits\displaystyle\sum_{u=1}^{q_{r_2}-q_l+i}\limits E^{(r_1)}_{q_{r_1}-q_l+i+1,q_{r_1}-q_{r_2}+u}t^{-s}E^{(r_2)}_{u,q_{r_2}-q_l+i+1}t^s\\
&\quad+\hbar\sum_{r_1<r_2}\limits\displaystyle\sum_{s \geq 0}\limits\displaystyle\sum_{u=q_{r_2}-q_l+i+1}^{q_{r_2}} \limits E^{(r_1)}_{q_{r_1}-q_{r_2}+u,q_{r_1}-q_{l}+i+1}t^{-s} E^{(r_2)}_{q_{r_2}-q_l+i+1,u}t^{s}\\
&\quad-\hbar\sum_{r_1<r_2}\limits\displaystyle\sum_{s \geq 0}\limits\displaystyle\sum_{u=q_{r_2}-q_l+i+1}^{q_{r_2}} \limits E^{(r_1)}_{q_{r_1}-q_l+i+1,q_{r_1}-q_{r_2}+u}t^{-s-1}E^{(r_2)}_{u,q_{r_2}-q_l+i+1}t^{s+1}
\end{align*}
and
\begin{align*}
C_i&=\hbar\sum_{r=1}^l\displaystyle\sum_{s \geq 0}\limits\sum_{u=1}^{q_r-q_l} E^{(r)}_{u,q_r-q_l+i}t^{-s-1}E^{(r)}_{q_r-q_l+i,u}t^{s+1}\\
&\quad-\hbar\displaystyle\sum_{s \geq 0}\limits\sum_{u=1}^{q_r-q_l} E^{(r)}_{u,q_r-q_l+i+1}t^{-s-1} E^{(r)}_{q_r-q_l+i+1,u}t^{s+1}\\
&\quad+\hbar\displaystyle\sum_{r_1<r_2}\sum_{s \geq 0}\limits\sum_{u=1}^{q_{r_2}-q_l} E^{(r_1)}_{q_{r_1}-q_{r_2}+u,q_{r_1}-q_l+i}t^{-s-1}E^{(r_2)}_{q_{r_2}-q_l+i,u}t^{s+1}\\
&\quad-\hbar\sum_{r_1<r_2}\displaystyle\sum_{s \geq 0}\limits\sum_{u=1}^{q_{r_2}-q_l} E^{(r_1)}_{q_{r_1}-q_{r_2}+u,q_{r_1}-q_l+i+1}t^{-s-1} E^{(r_2)}_{q_{r_2}-q_l+i+1,u}t^{s+1}\\
&\quad+\hbar\sum_{r_1<r_2}\displaystyle\sum_{s \geq 0}\limits\sum_{u=1}^{q_{r_2}-q_l} E^{(r_2)}_{u,q_{r_2}-q_l+i}t^{-s-1} E^{(r_1)}_{q_{r_1}-q_l+i,q_{r_1}-q_{r_2}+u}t^{s+1}\\
&\quad-\hbar\sum_{r_1<r_2}\displaystyle\sum_{s \geq 0}\limits\sum_{u=1}^{q_{r_2}-q_l} E^{(r_2)}_{u,q_{r_2}-q_l+i+1}t^{-s-1} E^{(r_1)}_{q_{r_1}-q_l+i+1,q_{r_1}-q_{r_2}+u}t^{s+1}.
\end{align*}
We note that $B_i$ comes from the coproduct for the affine Yangian and $C_i$ comes from the homomorphism $\Psi^{m.m+n}$.

By the definition of the evaluation map, we obtain
\begin{align}
&\quad\bigotimes_{1\leq s\leq l}\ev_{\hbar,\ve-(q_s-q_{Min})\hbar}^{q_s,(\gamma_s-\frac{q_s-q_{Min}}{2})\hbar}(\bigotimes_{1\leq s\leq l}H^{(s)}_{i+q_s-q_l})\nonumber\\
&=\sum_{r=1}^l\limits(\gamma_r-\dfrac{i+q_s-q_l}{2}\hbar)(e^{(r)}_{q_1-q_l+i,q_1-q_l+i} -e^{(r)}_{q_1-q_l+i+1,q_1-q_l+i+1})\nonumber\\
&\quad-\sum_{r=1}^l\limits\hbar e^{(r)}_{q_1-q_l+i,q_1-q_l+i}e^{(r)}_{q_1-q_l+i+1,q_1-q_l+i+1} \nonumber\\
&\quad+ \hbar \sum_{r=1}^l\limits\displaystyle\sum_{s \geq 0}  \limits\displaystyle\sum_{u=q_1-q_r+1}^{q_1-q_l+i}\limits  e^{(r)}_{q_1-q_l+i,u}t^{-s}e^{(r)}_{u,q_1-q_l+i}t^s\nonumber\\
&\quad+\hbar\sum_{r=1}^l\limits \displaystyle\sum_{s \geq 0} \limits\displaystyle\sum_{u=q_1-q_l+i+1}^{q_1}\limits  e^{(r)}_{q_1-q_l+i,u}t^{-s-1}e^{(r)}_{u,q_1-q_l+i}t^{s+1}\nonumber\\
&\quad-\hbar\sum_{r=1}^l\limits\displaystyle\sum_{s \geq 0}\limits\displaystyle\sum_{u=q_1-q_r+1}^{q_1-q_l+i}\limits e^{(r)}_{q_1-q_l+i+1,u}t^{-s} e^{(r)}_{u,q_1-q_l+i+1}t^{s}\nonumber\\
&\quad-\hbar\sum_{r=1}^l\limits\displaystyle\sum_{s \geq 0}\limits\displaystyle\sum_{u=q_1-q_l+i+1}^{q_1} \limits e^{(r)}_{q_1-q_l+i+1,u}t^{-s-1} e^{(r)}_{u,q_1-q_l+i+1}t^{s+1},\label{551}\\
&\quad\bigotimes_{1\leq s\leq l}\ev_{\hbar,\ve+(q_s-q_l)\hbar}^{q_s,(\gamma_s-\frac{q_s-q_{Min}}{2})\hbar}(B_i)\nonumber\\
&=-\hbar\sum_{r_1<r_2}\limits e^{(r_1)}_{q_1-q_l+i,q_1-q_l+i}e^{(r_2)}_{q_1-q_l+i+1,q_1-q_l+i+1}-\hbar\sum_{r_1<r_2}\limits e^{(r_1)}_{q_1-q_l+i+1,q_1-q_l+i+1}e^{(r_2)}_{q_1-q_l+i,q_1-q_l+i}\nonumber\\
&\quad-\hbar\sum_{r_1<r_2}\limits\displaystyle\sum_{s \geq 0}  \limits\displaystyle\sum_{u=1}^{q_{r_2}-q_l+i}\limits e^{(r_1)}_{q_{1}-q_{r_2}+u,q_{1}-q_l+i}t^{-s-1}e^{(r_2)}_{q_{1}-q_{r_2}+i,u}t^{s+1}\nonumber\\
&\quad+\hbar\sum_{r_1<r_2}\limits\displaystyle\sum_{s \geq 0}  \limits\displaystyle\sum_{u=1}^{q_{r_2}-q_l+i}\limits e^{(r_1)}_{q_{1}-q_l+i,q_{1}-q_{r_2}+u}t^{-s}e^{(r_2)}_{q_{1}-q_{r_2}+u,q_{1}-q_l+i}t^s\nonumber\\
&\quad-\hbar\sum_{r_1<r_2}\limits\displaystyle\sum_{s \geq 0} \limits\displaystyle\sum_{u=q_{r_2}-q_l+i+1}^{q_{r_2}}\limits e^{(r_1)}_{q_{1}-q_{r_2}+u,q_{1}-q_l+i}t^{-s}e^{(r_2)}_{q_{1}-q_l+i,q_{1}-q_{r_2}+u}t^{s}\nonumber\\
&\quad+\hbar\sum_{r_1<r_2}\limits\displaystyle\sum_{s \geq 0} \limits\displaystyle\sum_{u=q_{r_2}-q_l+i+1}^{q_{r_2}}\limits e^{(r_1)}_{q_{1}-q_l+i,q_{1}-q_{r_2}+u}t^{-s-1}e^{(r_2)}_{q_{1}-q_{r_2}+u,q_{1}-q_l+i}t^{s+1}\nonumber\\
&\quad+\hbar\sum_{r_1<r_2}\limits\displaystyle\sum_{s \geq 0}\limits\displaystyle\sum_{u=1}^{q_{r_2}-q_l+i}\limits e^{(r_1)}_{q_{1}-q_{r_2}+u,q_{1}-q_l+i+1}t^{-s-1}e^{(r_2)}_{q_{1}-q_l+i+1,q_{1}-q_{r_2}+u}t^{s+1}\nonumber\\
&\quad-\hbar\sum_{r_1<r_2}\limits\displaystyle\sum_{s \geq 0}\limits\displaystyle\sum_{u=1}^{q_{r_2}-q_l+i}\limits e^{(r_1)}_{q_{1}-q_l+i+1,q_{1}-q_{r_2}+u}t^{-s}e^{(r_2)}_{q_{1}-q_{r_2}+u,q_{1}-q_l+i+1}t^s\nonumber\\
&\quad+\hbar\sum_{r_1<r_2}\limits\displaystyle\sum_{s \geq 0}\limits\displaystyle\sum_{u=q_{r_2}-q_l+i+1}^{q_{r_2}} \limits e^{(r_1)}_{q_{1}-q_{r_2}+u,q_{1}-q_l+i+1}t^{-s} e^{(r_2)}_{q_{1}-q_l+i+1,q_{1}-q_{r_2}+u}t^{s}\nonumber\\
&\quad-\hbar\sum_{r_1<r_2}\limits\displaystyle\sum_{s \geq 0}\limits\displaystyle\sum_{u=q_{r_2}-q_l+i+1}^{q_{r_2}} \limits e^{(r_1)}_{q_{1}-q_l+i+1,q_{1}-q_{r_2}+u}t^{-s-1}e^{(r_2)}_{q_{1}-q_{r_2}+u,q_{1}-q_l+i+1}t^{s+1},\label{552}\\
&\quad\bigotimes_{1\leq s\leq l}\ev_{\hbar,\ve-(q_s-q_{Min})\hbar}^{q_s,(\gamma_s-\frac{q_s-q_{Min}}{2})\hbar}(C_i)\nonumber\\
&=\hbar\sum_{r=1}^l\displaystyle\sum_{s \geq 0}\limits\sum_{u=1}^{q_r-q_l} e^{(r)}_{q_1-q_r+u,q_1-q_l+i}t^{-s-1}e^{(r)}_{q_1-q_l+i,q_1-q_r+u}t^{s+1}\nonumber\\
&\quad-\hbar\displaystyle\sum_{s \geq 0}\limits\sum_{u=1}^{q_r-q_l} e^{(r)}_{q_1-q_r+u,q_1-q_l+i+1}t^{-s-1} e^{(r)}_{q_1-q_l+i+1,q_1-q_r+u}t^{s+1}\nonumber\\
&\quad+\hbar\displaystyle\sum_{r_1<r_2}\sum_{s \geq 0}\limits\sum_{u=1}^{q_{r_2}-q_l} e^{(r_1)}_{q_1-q_{r_2}+u,q_1-q_l+i}t^{-s-1}e^{(r_2)}_{q_1-q_l+i,q_1-q_{r_2}+u}t^{s+1}\nonumber\\
&\quad-\hbar\displaystyle\sum_{r_1<r_2}\displaystyle\sum_{s \geq 0}\limits\sum_{u=1}^{q_{r_2}-q_l} e^{(r_1)}_{q_1-q_{r_2}+u,q_1-q_l+i+1}t^{-s-1} e^{(r_2)}_{q_1-q_l+i+1,q_1-q_{r_2}+u}t^{s+1}\nonumber\\
&\quad+\hbar\displaystyle\sum_{r_1<r_2}\displaystyle\sum_{s \geq 0}\limits\sum_{u=1}^{q_{r_2}-q_l} e^{(r_2)}_{q_1-q_{r_2}+u,q_1-q_l+i}t^{-s-1} e^{(r_1)}_{q_1-q_l+i,q_1-q_{r_2}+u}t^{s+1}\nonumber\\
&\quad-\hbar\displaystyle\sum_{r_1<r_2}\displaystyle\sum_{s \geq 0}\limits\sum_{u=1}^{q_{r_2}-q_l} e^{(r_2)}_{q_1-q_{r_2}+u,q_1-q_l+i+1}t^{-s-1} e^{(r_1)}_{q_1-q_l+i+1,q_1-q_{r_2}+u}t^{s+1}.\label{553}
\end{align}
By the definition of $\widetilde{\mu}$, we have
\begin{align}
&\quad\hbar\widetilde{\mu}(W^{(2)}_{q_1-q_l+i,q_1-q_l+i}t)-\hbar\widetilde{\mu}(W^{(2)}_{q_1-q_l+i+1,q_1-q_l+i+1}t)\nonumber\\
&=\hbar\sum_{r=1}^n\gamma_re^{(r)}_{q_1-q_l+i,q_1-q_l+i}+\hbar\sum_{s\in\mathbb{Z}}\limits\sum_{r_1<r_2}\sum_{u>q_1-q_l}\limits e^{(r_1)}_{u,q_1-q_l+i}t^{-s}e^{(r_2)}_{q_1-q_l+i,u}t^s\nonumber\\
&\quad-\hbar\sum_{s\geq0}\limits\sum_{r\geq0}\limits\sum_{1\leq u \leq q_r-q_l}\limits e^{(r)}_{q_1-q_r+u,q_1-q_l+i}t^{-s-1}e^{(r)}_{q_1-q_l+i,q_1-q_r+u}t^{s+1}\nonumber\\
&\quad-\hbar\sum_{s\geq0}\limits\sum_{r\geq0}\limits\sum_{1\leq u \leq q_r-q_l}\limits e^{(r)}_{q_1-q_l+i,q_1-q_r+u}t^{-s}e^{(r)}_{q_1-q_r+u,q_1-q_l+i}t^{s}\nonumber\\
&\quad-\hbar\sum_{s\in\mathbb{Z}}\limits\sum_{r_1<r_2}\limits\sum_{1\leq u \leq q_{r_2}-q_l}\limits e^{(r_1)}_{q_1-q_l+i,q_1-q_{r_2}+u}t^{s+1}e^{(r_2)}_{q_1-q_{r_2}+u,q_1-q_l+i}t^{-s-1}\nonumber\\
&\quad-\hbar\sum_{r=1}^n\gamma_re^{(r)}_{q_1-q_l+i+1,q_1-q_l+i+1}-\hbar\sum_{s\in\mathbb{Z}}\limits\sum_{r_1<r_2}\sum_{u>q_1-q_l}\limits e^{(r_1)}_{u,q_1-q_l+i+1}t^{-s}e^{(r_2)}_{q_1-q_l+i+1,u}t^s\nonumber\\
&\quad+\hbar\sum_{s\geq0}\limits\sum_{r\geq0}\limits\sum_{1\leq u \leq q_r}\limits e^{(r)}_{q_1-q_r+u,q_1-q_l+i+1}t^{-s-1}e^{(r)}_{q_1-q_l+i+1,q_1-q_r+u}t^{s+1}\nonumber\\
&\quad+\hbar\sum_{s\geq0}\limits\sum_{r\geq0}\limits\sum_{1\leq u \leq q_r}\limits e^{(r)}_{q_1-q_l+i+1,q_1-q_r+u}t^{-s}e^{(r)}_{q_1-q_r+u,q_1-q_l+i+1}t^{s}\nonumber\\
&\quad+\hbar\sum_{s\in\mathbb{Z}}\limits\sum_{r_1<r_2}\limits\sum_{1\leq u \leq q_{r_2}-q_l}\limits e^{(r_1)}_{q_1-q_l+i+1,q_1-q_{r_2}+u}t^{s+1}e^{(r_2)}_{q_1-q_{r_2}+u,q_1-q_l+i+1}t^{-s-1}.\label{554}
\end{align}
Then, by a direct computation, the sum of \eqref{551}-\eqref{554} is equal to 
\begin{equation*}
\bigotimes_{1\leq s\leq l}\ev_{\hbar,\ve-(q_s-q_{Min})\hbar}^{q_s,(\gamma_s-\frac{q_s-q_l}{2})\hbar}\circ\Delta^l(H_{i,1})+\hbar\widetilde{\mu}(W^{(2)}_{q_1-q_l+i,q_1-q_l+i}t)-\hbar\widetilde{\mu}(W^{(2)}_{q_1-q_l+i+1,q_1-q_l+i+1}t).
\end{equation*}
We divide the sum into 8 piecies:
\begin{align}
&\quad\eqref{551}_1+\eqref{554}_1+\eqref{554}_6=-\dfrac{i}{2}\hbar \widetilde{\mu}(W^{(1)}_{q_1-q_l+i,q_1-q_l+i}-W^{(1)}_{q_1-q_l+i+1,q_1-q_l+i+1}),\\
&\quad\eqref{551}_2+\eqref{552}_1+\eqref{552}_2=-\hbar\widetilde{\mu}(W^{(1)}_{q_1-q_l+i,q_1-q_l+i}W^{(1)}_{q_1-q_l+i+1,q_1-q_l+i+1}),\\
&\quad\eqref{551}_3+\eqref{551}_4+\eqref{553}_1+\eqref{554}_3+\eqref{554}_4\nonumber\\
&=\hbar \sum_{r=1}^l\limits\displaystyle\sum_{s \geq 0}  \limits\displaystyle\sum_{u=q_1-q_l+1}^{q_1-q_l+i}\limits  e^{(r)}_{q_1-q_l+i,u}t^{-s}e^{(r)}_{u,q_1-q_l+i}t^s\nonumber\\
&\quad+\hbar\sum_{r=1}^l\limits \displaystyle\sum_{s \geq 0} \limits\displaystyle\sum_{u=q_1-q_l+i+1}^{q_1}\limits  e^{(r)}_{q_1-q_l+i,u}t^{-s-1}e^{(r)}_{u,q_1-q_l+i}t^{s+1},\label{551--1}\\
&\quad\eqref{551}_5+\eqref{551}_6+\eqref{553}_2+\eqref{554}_{8}+\eqref{554}_{9}\nonumber\\
&=-\hbar\sum_{r=1}^l\limits\displaystyle\sum_{s \geq 0}\limits\displaystyle\sum_{u=q_1-q_l+1}^{q_1-q_l+i}\limits e^{(r)}_{q_1-q_l+i+1,u}t^{-s} e^{(r)}_{u,q_1-q_l+i+1}t^{s}\nonumber\\
&\quad-\hbar\sum_{r=1}^l\limits\displaystyle\sum_{s \geq 0}\limits\displaystyle\sum_{u=q_1-q_l+i+1}^{n} \limits e^{(r)}_{q_1-q_l+i+1,u}t^{-s-1} e^{(r)}_{u,q_1-q_l+i+1}t^{s+1},\label{551-0}\\
&\quad\eqref{552}_3+\eqref{552}_5+\eqref{553}_3+\eqref{554}_{2}\nonumber\\
&=\hbar\sum_{r_1<r_2}\limits\displaystyle\sum_{s \geq 0}  \limits\displaystyle\sum_{u=q_{r_2}-q_l+1}^{q_{r_2}-q_l+i}\limits e^{(r_1)}_{q_{1}-q_{r_2}+u,q_{1}-q_l+i}t^{s}e^{(r_2)}_{q_{1}-q_{r_2}+i,u}t^{-s}\nonumber\\
&\quad+\hbar\sum_{r_1<r_2}\limits\displaystyle\sum_{s \geq 0} \limits\displaystyle\sum_{u=q_{r_2}-q_l+i+1}^{q_{r_2}}\limits e^{(r_1)}_{q_{1}-q_{r_2}+u,q_{1}-q_l+i}t^{s+1}e^{(r_2)}_{q_{1}-q_l+i,q_{1}-q_{r_2}+u}t^{-s-1},\label{551-1}\\
&\quad\eqref{552}_4+\eqref{552}_6+\eqref{553}_5+\eqref{554}_{5}\nonumber\\
&=\hbar\sum_{r_1<r_2}\limits\displaystyle\sum_{s \geq 0}  \limits\displaystyle\sum_{u=q_{r_2}-q_l+1}^{q_{r_2}-q_l+i}\limits e^{(r_1)}_{q_{1}-q_l+i,q_{1}-q_{r_2}+u}t^{-s}e^{(r_2)}_{q_{1}-q_{r_2}+u,q_{1}-q_l+i}t^s\nonumber\\
&\quad+\hbar\sum_{r_1<r_2}\limits\displaystyle\sum_{s \geq 0} \limits\displaystyle\sum_{u=q_{r_2}-q_l+i+1}^{q_{r_2}}\limits e^{(r_1)}_{q_{1}-q_l+i,q_{1}-q_{r_2}+u}t^{-s-1}e^{(r_2)}_{q_{1}-q_{r_2}+u,q_{1}-q_l+i}t^{s+1},\label{551-2}\\
&\quad\eqref{552}_7+\eqref{552}_9+\eqref{553}_4+\eqref{554}_{7}\nonumber\\
&=-\hbar\sum_{r_1<r_2}\limits\displaystyle\sum_{s \geq 0}\limits\displaystyle\sum_{u=q_{r_2}-q_l+1}^{q_{r_2}-q_l+i}\limits e^{(r_1)}_{q_{1}-q_{r_2}+u,q_{1}-q_l+i+1}t^{s}e^{(r_2)}_{q_{1}-q_l+i+1,q_{1}-q_{r_2}+u}t^{-s}\nonumber\\
&\quad-\hbar\sum_{r_1<r_2}\limits\displaystyle\sum_{s \geq 0}\limits\displaystyle\sum_{u=q_{r_2}-q_l+i+1}^{q_{r_2}} \limits e^{(r_1)}_{q_{1}-q_{r_2}+u,q_{1}-q_l+i+1}t^{s+1} e^{(r_2)}_{q_{1}-q_l+i+1,q_{1}-q_{r_2}+u}t^{-s-1},\label{551-3}\\
&\quad\eqref{552}_8+\eqref{552}_{10}+\eqref{553}_6+\eqref{554}_{10}\nonumber\\
&=-\hbar\sum_{r_1<r_2}\limits\displaystyle\sum_{s \geq 0}\limits\displaystyle\sum_{u=q_{r_2}-q_l+1}^{q_{r_2}-q_l+i}\limits e^{(r_1)}_{q_{1}-q_l+i+1,q_{1}-q_{r_2}+u}t^{-s}e^{(r_2)}_{q_{1}-q_{r_2}+u,q_{1}-q_l+i+1}t^s\nonumber\\
&\quad-\hbar\sum_{r_1<r_2}\limits\displaystyle\sum_{s \geq 0}\limits\displaystyle\sum_{u=q_{r_2}-q_l+i+1}^{q_{r_2}} \limits e^{(r_1)}_{q_{1}-q_l+i+1,q_{1}-q_{r_2}+u}t^{-s-1}e^{(r_2)}_{q_{1}-q_{r_2}+u,q_{1}-q_l+i+1}t^{s+1}.\label{551-4}
\end{align}
By a direct computation, we have
\begin{align*}
\eqref{551--1}+\eqref{551-1}+\eqref{551-2}
&=\widetilde{\mu}(\hbar\displaystyle\sum_{s \geq 0}  \limits\displaystyle\sum_{u=1}^{i}\limits W^{(1)}_{q_1-q_l+i,q_1-q_l+u}t^{-s}W^{(1)}_{q_1-q_l+u,q_1-q_l+i}t^s\\
&\quad+\hbar\displaystyle\sum_{s \geq 0} \limits\displaystyle\sum_{u=i+1}^{n}\limits W^{(1)}_{q_1-q_l+i,q_1-q_l+u}t^{-s-1} W^{(1)}_{q_1-q_l+u,q_1-q_l+i}t^{s+1})
\end{align*}
and
\begin{align*}
\eqref{551-0}+\eqref{551-3}+\eqref{551-4}
&=\widetilde{\mu}(-\hbar\displaystyle\sum_{s \geq 0}\limits\displaystyle\sum_{u=1}^{i}\limits W^{(1)}_{q_1-q_l+i+1,q_1-q_l+u}t^{-s} W^{(1)}_{q_1-q_l+u,q_1-q_l+i+1}t^s\\
&\quad-\hbar\displaystyle\sum_{s \geq 0}\limits\displaystyle\sum_{u=i+1}^{n} \limits W^{(1)}_{q_1-q_l+i+1,q_1-q_l+u}t^{-s-1} W^{(1)}_{q_1-q_l+u,q_1-q_l+i+1}t^{s+1}).
\end{align*}
Thus, we have proven the relation \eqref{112}.
\end{proof}
\section{Coproduct for the affine Yangian and the parabolic induction for $W$-algebras}
In this section, we show the compatibility with the coproduct for the affine Yangian and the parabolic induction for a $W$-algebra of type $A$. Here after, we assume that $\hbar\neq0$. Let us take $N_1,N_2$ and nilpotent elements of $f_1,f_2$ in the same way as Section 7. By Theorem~\ref{Main}, we obtain the following theorem.
\begin{Theorem}
\begin{enumerate}
\item Suppose that $\ve=\hbar(k+N-q_{Min})$ and $q_1\geq\min\{q_w,q_{w+1}\}$. Then, we can define a homomorphism
\begin{gather*}
\Phi_1\colon Y_{\hbar,\ve-(\min\{q_1,q_w\}-q_{Min})\hbar}(\widehat{\mathfrak{sl}}(\min\{q_1,q_w\}))\to \mathcal{U}(\mathcal{W}^{k+N_2}(\mathfrak{gl}(N_1),f_1))
\end{gather*}
by the same formula as $\Phi$.
\item Suppose that $\ve=\hbar(k+N-q_{Min})$ and $q_l\geq\min\{q_{w+1},q_w\}$. Then, we can define a homomorphism
\begin{gather*}
\Phi_2\colon Y_{\hbar,\ve-(\min\{q_{w+1},q_l\}-q_{Min})\hbar}(\widehat{\mathfrak{sl}}(\min\{q_{w+1},q_l\}))\to \mathcal{U}(\mathcal{W}^{k+N_1}(\mathfrak{gl}(N_2),f_2))
\end{gather*}
by the same formula as $\Phi$.
\end{enumerate}
\end{Theorem}
In order to consider the case that $q_1<q_w,q_{w+1}$ or $q_l<q_w,q_{w+1}$, we need to construct new associative algebras. 

For integers $n\geq0$ and $1\leq i\leq m+n$, we set the following elements in $\widetilde{Y}_{\hbar,\ve}(\widehat{\mathfrak{sl}}(m+n))$:
\begin{align*}
v_i
&=\dfrac{\hbar}{2}\sum_{\substack{i<z\leq m+n,\\s\geq0}}\limits E_{z,i}t^{-s}E_{i,z}t^s+\dfrac{\hbar}{2}\sum_{\substack{i>z\geq1,\\s\geq0}}\limits E_{z,i}t^{-s-1}E_{i,z}t^{s+1}\\
&\quad-\dfrac{\hbar}{2}\sum_{\substack{1\leq u<i,\\s\geq0}}\limits E_{i,u}t^{-s}E_{u,i}t^s-\dfrac{\hbar}{2}\sum_{\substack{m+n\geq u>i,\\s\geq0}}\limits E_{i,u}t^{-s-1}E_{u,i}t^{s+1},\\
J(h_i)&=\widetilde{H}_{i,1}+v_i-v_{i+1}.
\end{align*}
Let $\{\alpha_u\}_{u=1}^{m+n}$ be simple roots of $\widehat{\mathfrak{sl}}(m+n)$, $\Delta^{\text{re}}$ be a set of real positive roots of $\widehat{\mathfrak{sl}}(m+n)$ and $(\ ,\ )$ be the standard form on $\widehat{\mathfrak{sl}}(m+n)$.  

In Proposition 3.21 of \cite{GNW}, the following relation was obtained:
\begin{align}
(\alpha_j,\beta)[J(h_i),x_\beta]-(\alpha_i,\beta)[J(h_j),x_\beta]\in\mathbb{C}x_\beta,\label{JJ}
\end{align}
where $x_\beta$ is a root vector whose root is $\beta$. 

Assume that $\widetilde{\ve}=\ve-n\hbar$ and $\Psi^{m.m+n}$ is an edge contraction from $Y_{\hbar,\ve}(\widehat{\mathfrak{sl}}(m))$ to $\widetilde{Y}_{\hbar,\widetilde{\ve}}(\widehat{\mathfrak{sl}}(m+n))$.
By using \eqref{JJ}, for $1\leq v\leq n$ and $1\leq j\leq m$, the following relations hold in  $\widetilde{Y}_{\hbar,\widetilde{\ve}}(\widehat{\mathfrak{sl}}(m+n))$:
\begin{align*}
&\quad[\Psi^{m,m+n}(\widetilde{H}_{i,1}),E_{v,n+j}t^x]=[\ev^{m+n,0}_{\hbar,\widetilde{\ve}}(\Psi^{m,m+n}(\widetilde{H}_{i,1})),E_{v,n+j}t^x]\\
&=a^{m,m+n}_{i,j}-a^{m,m+n}_{i+1,j}\text{ for }j\neq i,i+1,\\
&\quad[\Psi^{m,m+n}(\widetilde{H}_{i,1}),E_{n+j,v}t^{-x}]
=[\ev^{m+n,0}_{\hbar,\widetilde{\ve}}(\Psi^{m,m+n}(\widetilde{H}_{i,1})),E_{n+j,v}t^{-x}]\\
&=b^{m,m+n}_{i,j}-b^{m,m+n}_{i+1,j}\text{ for }j\neq i,i+1,\\
&\quad[\Psi^{m,m+n}(\widetilde{H}_{0,1}),E_{v,n+j}t^x]
=[\ev^{m+n,0}_{\hbar,\widetilde{\ve}}(\Psi^{m,m+n}(\widetilde{H}_{0,1})),E_{v,n+j}t^x]\\
&=a^{m,m+n}_{m,j}-a^{m,m+n}_{1,j}\text{ for }j\neq 1,m\\
&\quad[\Psi^{m,m+n}(\widetilde{H}_{0,1}),E_{n+j,v}t^{-x}]
=[\ev^{m+n,0}_{\hbar,\widetilde{\ve}}(\Psi^{m,m+n}(\widetilde{H}_{0,1})),E_{n+j,v}t^{-x}]\\
&=b^{m,m+n}_{m,j}-b^{m,m+n}_{1,j}\text{ for }j\neq 1,m,\\
&\quad[\Psi^{m,m+n}(\widetilde{H}_{i-1,1}),E_{v,n+i}t^x]+[\Psi^{m,m+n}(\widetilde{H}_{i,1}),E_{v,n+i}t^x]\\
&=[\ev^{m+n,0}_{\hbar,\widetilde{\ve}}(\Psi^{m,m+n}(\widetilde{H}_{i-1,1})),E_{v,n+i}t^x]+[\ev^{m+n,0}_{\hbar,\widetilde{\ve}}(\Psi^{m,m+n}(\widetilde{H}_{i,1})),E_{v,n+i}t^x]\\
&=a^{m,m+n}_{i-1,i}-a^{m,m+n}_{i+1,i}-\dfrac{\hbar}{2}E_{v,n+i}t^x,\\
&\quad[\Psi_{m,m+n}(\widetilde{H}_{i-1,1}),E_{n+i,v}t^{-x}]+[\Psi^{m,m+n}(\widetilde{H}_{i,1}),E_{n+i,v}t^{-x}]\\
&=[\ev^{m+n,0}_{\hbar,\widetilde{\ve}}(\Psi_{m,m+n}(\widetilde{H}_{i-1,1})),E_{n+i,v}t^{-x}]+[\ev^{m+n,0}_{\hbar,\widetilde{\ve}}(\Psi^{m,m+n}(\widetilde{H}_{i,1})),E_{n+i,v}t^{-x}]\\
&=b^{m,m+n}_{i-1,i}-b^{m,m+n}_{i+1,i}+\dfrac{\hbar}{2}E_{n+i,v}t^{-x},\\
&\quad[\Psi^{m,m+n}(\widetilde{H}_{0,1}),e_{v,n+1}t^{x}]+[\Psi^{m,m+n}(\widetilde{H}_{1,1}),E_{v,n+1}t^x]\\
&=[\ev^{m+n,0}_{\hbar,\widetilde{\ve}}(\Psi^{m,m+n}(\widetilde{H}_{0,1})),e_{v,n+1}t^{x}]+[\ev^{m+n,0}_{\hbar,\widetilde{\ve}}(\Psi^{m,m+n}(\widetilde{H}_{1,1})),E_{v,n+1}t^{x}]\\
&=a^{m,m+n}_{m,1}-a^{m,m+n}_{2,1}+\dfrac{\hbar}{2}E_{v,n+1}t^x,\\
&\quad[\Psi^{m,m+n}(\widetilde{H}_{0,1}),E_{n+1,v}t^{-x}]+[\Psi^{m,m+n}(\widetilde{H}_{1,1}),E_{n+1,v}t^{-x}]\\
&=[\ev^{m+n,0}_{\hbar,\widetilde{\ve}}(\Psi^{m,m+n}(\widetilde{H}_{0,1})),E_{n+1,v}t^{-x}]+[\ev^{m+n,0}_{\hbar,\widetilde{\ve}}(\Psi^{m,m+n}(\widetilde{H}_{1,1})),E_{n+1,v}t^{-x}]\\
&=b^{m,m+n}_{m,1}-b^{m,m+n}_{2,1}-\dfrac{\hbar}{2}E_{n+1,v}t^{-x},\\
&\quad
[\Psi^{m,m+n}(\widetilde{H}_{0,1}),E_{v,m+n}t^w]+[\Psi^{m,m+n}(\widetilde{H}_{m-1,1}),E_{v,m+n}t^w]\\
&=[\ev^{m+n,0}_{\hbar,\widetilde{\ve}}(\Psi^{m,m+n}(\widetilde{H}_{0,1})),E_{v,m+n}t^w]+[\ev^{m+n,0}_{\hbar,\widetilde{\ve}}(\Psi^{m,m+n}(\widetilde{H}_{m-1,1})),E_{v,m+n}t^w]\\
&=a^{m,m+n}_{1,m}-a^{m,m+n}_{m-1,m}-\dfrac{\hbar}{2}(m-1)E_{v,m+n}t^x-{\ve} E_{v,m+n}t^x,\\
&\quad[\Psi^{m,m+n}(\widetilde{H}_{0,1}),E_{m+n,v}t^{-x}]+[\Psi^{m,m+n}(\widetilde{H}_{m-1,1}),E_{m+n,v}t^{-x}]\\
&=[\ev^{m+n,0}_{\hbar,\widetilde{\ve}}(\Psi^{m,m+n}(\widetilde{H}_{0,1})),E_{m+n,v}t^{-x}]+[\ev^{m+n,0}_{\hbar,\widetilde{\ve}}(\Psi^{m,m+n}(\widetilde{H}_{m-1,1})),E_{m+n,v}t^{-x}]\\
&=b^{m,m+n}_{1,m}-b^{m,m+n}_{m-1,m}+\dfrac{\hbar}{2}(m-1)E_{m+n,v}t^{-x}+{\ve} E_{m+n,v}t^{-x},
\end{align*}
where
\begin{align*}
a^{m,m+n}_{i,j}&=
\delta(j<i)\hbar\displaystyle\sum_{s \geq 0}\limits E_{v,n+i}t^{x-s-1}E_{n+i,n+j}t^{s+1}+\delta(j>i)\hbar\sum_{s\geq0}\limits E_{v,n+i}t^{x-s}E_{n+i,n+j}t^{s},\\
b^{m,m+n}_{i,j}&=
\delta(j<i)\hbar\displaystyle\sum_{s \geq 0}\limits E_{n+j,n+i}t^{-s-1}E_{n+i,v}t^{-x+s+1}+\delta(j>i)\hbar\sum_{s\geq0}\limits E_{n+j,n+i}t^{-s}E_{n+i,v}t^{-x+s}.
\end{align*}
\begin{Definition}
For $n\geq 0$, we define $Y^{m+n}_{\hbar,\ve}(\widehat{\mathfrak{sl}}(m))$ as the associative algebra whose generators are $\{H_{i,r},X^\pm_{i,r}\mid0\leq i\leq m-1,r\in\mathbb{Z}_{\geq0}\}$ and $\widehat{\mathfrak{sl}}(m+n)=\mathfrak{sl}(m+n)\otimes\mathbb{C}[t^{\pm1}]\oplus\mathbb{C}\widetilde{c}$ with the relations \eqref{Eq2.1}-\eqref{Eq2.10} and
\begin{gather*}
H_{i,0}=\begin{cases}
E_{i+n,i+n}-E_{i+1+n,i+1+n}\text{ if }1\leq i\leq m-1,\\
E_{m+n,m+n}-E_{n+1,n+1}+\widetilde{c}\text{ if }i=0,
\end{cases}\\
X^+_{i,0}=\begin{cases}
E_{i+n,i+1+n}\text{ if }1\leq i\leq m-1,\\
E_{m+n,1+n}t\text{ if }i=0,
\end{cases}X^-_{i,0}=\begin{cases}
E_{i+1+n,i+n}\text{ if }1\leq i\leq m-1,\\
E_{1+n,m+n}t^{-1}\text{ if }i=0,
\end{cases}
\end{gather*}
\end{Definition}
We set the degree on $Y^{m+n}_{\hbar,\ve}(\widehat{\mathfrak{sl}}(m))$ as
\begin{equation*}
\text{deg}(H_{i,r})=0,\text{deg}(X^\pm_{i,r})=\pm \delta_{i,0},\text{deg}(zt^s)=s,\text{deg}(\widetilde{c})=0\text{ for }z\in\mathfrak{sl}(m+n).
\end{equation*}
We denote the standard degreewise completion of $Y^{m+n}_{\hbar,\ve}(\widehat{\mathfrak{sl}}(m))$ by $\widehat{Y}^{m+n}_{\hbar,\ve}(\widehat{\mathfrak{sl}}(m))$ .

We set an associative algebra $Y^{m+n,R}_{\hbar,\ve}(\widehat{\mathfrak{sl}}(m))$ as a quotient algebra of $\widehat{Y}^{m+n}_{\hbar,\ve}(\widehat{\mathfrak{sl}}(m))$ divided by
\begin{align}
[\widetilde{H}_{i,1},E_{v,n+j}t^x]
&=a^{m,m+n}_{i,j}-a^{m,m+n}_{i+1,j}\text{ for }j\neq i,i+1,\label{Eq2.11}\\
[\widetilde{H}_{i-1,1},E_{v,n+i}t^x]+[\widetilde{H}_{i,1},E_{v,n+i}t^x]
&=a^{m,m+n}_{i-1,i}-a^{m,m+n}_{i+1,i}-\dfrac{\hbar}{2}E_{v,n+i}t^x,\label{Eq2.13}\\
[\widetilde{H}_{0,1},E_{v,n+j}t^x]
&=a^{m,m+n}_{m,j}-a^{m,m+n}_{1,j}\text{ for }j\neq 1,m\label{Eq2.15}\\
[\widetilde{H}_{0,1},e_{v,n+1}t^x]+[\widetilde{H}_{1,1},E_{v,n+1}t^x]
&=a^{m,m+n}_{m,1}-a^{m,m+n}_{2,1}+\dfrac{\hbar}{2}E_{v,n+1}t^x,\label{Eq2.17}\\
[\widetilde{H}_{0,1},E_{v,m+n}t^x]+[\widetilde{H}_{m-1,1},E_{v,m+n}t^x]
&=-a^{m,m+n}_{1,m}+a^{m,m+n}_{m-1,m}-\dfrac{\hbar}{2}(m-1)E_{v,m+n}t^x-{\ve} E_{v,m+n}t^x.\label{Eq2.19}
\end{align}
We also set an associative algebra $Y^{m+n,L}_{\hbar,\ve}(\widehat{\mathfrak{sl}}(m))$ as a quotient algebra of $\widehat{Y}^{m+n}_{\hbar,\ve}(\widehat{\mathfrak{sl}}(m))$ divided by
\begin{align}
[\widetilde{H}_{i,1},E_{n+j,v}t^{-x}]
&=b^{m,m+n}_{i,j}-b^{m,m+n}_{i+1,j}\text{ for }j\neq i,i+1,\label{Eq2.12}\\
[\widetilde{H}_{i-1,1},E_{n+i,v}t^{-x}]+[\widetilde{H}_{i,1},E_{n+i,v}t^{-x}]
&=b^{m,m+n}_{i-1,i}-b^{m,m+n}_{i+1,i}+\dfrac{\hbar}{2}E_{n+i,v}t^{-x},\label{Eq2.14}\\
[\widetilde{H}_{0,1},E_{n+j,v}t^{-x}]
&=b^{m,m+n}_{m,j}-b^{m,m+n}_{1,j}\text{ for }j\neq 1,m\label{Eq2.16}\\
[\widetilde{H}_{0,1},E_{n+1,v}t^{-x}]+[\widetilde{H}_{1,1},E_{n+1,v}t^{-x}]
&=b^{m,m+n}_{m,1}-b^{m,m+n}_{2,1}-\dfrac{\hbar}{2}E_{n+1,v}t^{-x},\label{Eq2.18}\\
[\widetilde{H}_{0,1},E_{m+n,v}t^{-x}]&+[\widetilde{H}_{m-1,1},E_{m+n,v}t^{-x}]\nonumber\\
=-b^{m,m+n}_{1,m}&+b^{m,m+n}_{m-1,m}+\dfrac{\hbar}{2}(m-1)E_{m+n,v}t^{-x}+{\ve} E_{m+n,v}t^{-x}.\label{Eq2.20}
\end{align}
By the definition, $Y^{m+n,L}_{\hbar,\ve}(\widehat{\mathfrak{sl}}(m+n))$ and $Y^{m+n,L}_{\hbar,\ve}(\widehat{\mathfrak{sl}}(m+n))$ coincides with $\widetilde{Y}_{\hbar,\ve}(\widehat{\mathfrak{sl}}(m+n))$. For $p>m+n$, we define $Y^{m+n,L}_{\hbar,\ve}(\widehat{\mathfrak{sl}}(p))$.

The following theorem follows from the definition of $Y^{m+n,L}_{\hbar,\ve}(\widehat{\mathfrak{sl}}(m))$ and $Y^{m+n,R}_{\hbar,\ve}(\widehat{\mathfrak{sl}}(m))$.
\begin{Theorem}
For $m_1<m_2$, there exist homomorphisms
\begin{gather*}
\Psi_L^{m_1,m_2}\colon Y^{m_3,L}_{\hbar,\ve}(\widehat{\mathfrak{sl}}(m_1))\to Y^{m_3,L}_{\hbar,\ve-(m_2-m_1)\hbar}(\widehat{\mathfrak{sl}}(m_2)),\\
\Psi_R^{m_1,m_2}\colon Y^{m_3,R}_{\hbar,\ve}(\widehat{\mathfrak{sl}}(m_1))\to Y^{m_3,R}_{\hbar,\ve-(m_2-m_1)\hbar}(\widehat{\mathfrak{sl}}(m_2))
\end{gather*}
defined by the same formula as $\Psi^{m_1,m_2}$.
\end{Theorem}
\begin{Theorem}
Assume that $q_l<q_w,q_{w+1}$ and $\ve=\hbar(k+N-q_l)$. We can construct a homomorphism
\begin{align*}
\Phi_2&\colon Y^{\min\{q_w,q_{w+1}\},R}_{\hbar,\ve}(\widehat{\mathfrak{sl}}(q_l))\to\mathcal{U}(\mathcal{W}^{k+N_1}(\mathfrak{gl}(N_2),f))
\end{align*}
by the same formula as $\Phi$.
\end{Theorem}
\begin{proof}
The compatibility of $\Phi_1$ with \eqref{Eq2.1}-\eqref{Eq2.10} follows from Theorem~\ref{Main}. Thus, it is enough to show the compatibility with \eqref{Eq2.11}-\eqref{Eq2.19}. Let $q_{Max}^R$ be $\max\{q_i\mid w+1\leq i\leq l\}$. Then, we can write down
\begin{align*}
\Phi^R(\widetilde{H}_{i,1})&=-\dfrac{i}{2}\hbar(W^{(1)}_{q_{Max}^R-q_1+i,q_{Max}^R-q_1+i}-W^{(1)}_{q_{Max}^R-q_1+i+1,q_{Max}^R-q_1+i+1})+U_i-U_{i+1},\\
\Phi^R(\widetilde{H}_{0,1})&=\ve W^{(1)}_{q_{Max}^R,q_{Max}^R}-\sum_{y=1}^{q^R_{Max}-q_l}W^{(1)}_{y,y}+\hbar(\sum_{v=w+1}^{l-1}\alpha_v)-\dfrac{\hbar}{2}(\sum_{u=w+1}^l\alpha_u)^2+U_{q_1}-U_{1},
\end{align*}
where 
\begin{align*}
U_i&=-\hbar W^{(2)}_{q_{Max}^R-q_1+i,q_{Max}^R-q_1+i}t-\dfrac{\hbar}{2}((W^{(1)}_{q_{Max}^R-q_1+i,q_{Max}^R-q_{1}+i})^2\\
&\quad+\hbar\displaystyle\sum_{s \geq 0}  \limits\displaystyle\sum_{u=1}^{i}\limits W^{(1)}_{q_{Max}^R-q_1+i,q_{Max}^R-q_1+u}t^{-s}W^{(1)}_{q_{Max}^R-q_1+u,q_{Max}^R-q_1+i}t^s\\
&\quad+\hbar\displaystyle\sum_{s \geq 0} \limits\displaystyle\sum_{u=i+1}^{q_1}\limits W^{(1)}_{q_{Max}^R-q_1+i,q_{Max}^R-q_1+u}t^{-s-1} W^{(1)}_{q_{Max}^R-q_1+u,q_{Max}^R-q_1+i}t^{s+1}.
\end{align*}
By Lemma~\ref{OPE}, for $i\neq j$ and $v\leq q_{Max}^R-q_1$, we obtain
\begin{align*}
&\quad[U_i,W^{(1)}_{q_{Max}^R-q_1+j,v}t^{-x}]-\hbar [W^{(2)}_{q_{Max}^R-q_1+i,q_{Max}^R-q_1+i}t,W^{(1)}_{q_{Max}^R-q_1+j,v}t^{-x}]\\
&=-\delta(j<i)\hbar\sum_{s\geq0}\limits (W^{(1)}_{q_{Max}^R-q_1+j,q_{Max}^R-q_1+i}t^{-s-1}W^{(1)}_{q_{Max}^R-q_1+i,v}t^{s-x+1})\\
&\quad-\delta(j>i)\hbar\sum_{s\geq0}\limits (W^{(1)}_{q_{Max}^R-q_1+j,q_{Max}^R-q_1+i}t^{-s}W^{(1)}_{q_{Max}^R-q_1+i,v}t^{s-x})+\delta(j>i)W^{(1)}_{q_{Max}^R-q_1+j,v}t^{-x}.
\end{align*}

By a direct computation, we have proved the compatibility with \eqref{Eq2.11}-\eqref{Eq2.19}.
\end{proof}
Similarly, for $q_1<q_w,q_{w+1}$, we can construct a homomorphism
\begin{align*}
\Phi_1&\colon Y^{\min\{q_w,q_{w+1}\},L}_{\hbar,\ve}(\widehat{\mathfrak{sl}}(q_1))\to\mathcal{U}(\mathcal{W}^{k+N_2}(\mathfrak{gl}(N_1),f))
\end{align*}
by the same formula as $\Phi$.
\begin{Theorem}
We can construct a homomorphism
\begin{align*}
\Delta^{m}\colon  Y_{\hbar,\ve}(\widehat{\mathfrak{sl}}(m))\to Y^{m+n,L}_{\hbar,\ve}(\widehat{\mathfrak{sl}}(m))\widehat{\otimes}Y^{m+n,R}_{\hbar,\ve}(\widehat{\mathfrak{sl}}(p_1))
\end{align*}
determined by
\begin{align*}
\Delta^{m}(y)&=1\otimes y+y\otimes 1\text{ for }y\in\widehat{\mathfrak{sl}}(p_x),\\
\Delta^{m}(H_{i,1})&=\begin{cases}
H_{i,1}\otimes1+1\otimes H_{i,1}+A^+_i\\
\quad+\hbar\displaystyle\sum_{\substack{s\in\mathbb{Z}\\1\leq u\leq n}}\limits E_{i+n,u}t^{-s}\otimes E_{u,i+n}t^s-\hbar\displaystyle\sum_{\substack{s\in\mathbb{Z}\\1\leq u\leq n}}\limits E_{i+1+n,u}t^{-s}\otimes E_{u,i+1+n}t^s\text{ if }i\neq0,\\
H_{0,1}\otimes1+1\otimes H_{0,1}+A^+_0\\
\quad+\hbar\displaystyle\sum_{\substack{s\in\mathbb{Z}\\1\leq u\leq n}}\limits E_{m+n,u}t^{-s}\otimes E_{u,m+n}t^s-\hbar\displaystyle\sum_{\substack{s\in\mathbb{Z}\\1\leq u\leq n}}\limits E_{n+1,u}t^{-s}\otimes E_{u,n+1}t^s\text{ if }i=0.
\end{cases}
\end{align*}
\end{Theorem}
\begin{proof}
The compatibility with the defining relations of the affine Yangian except of $[H_{i,1},H_{j,1}]=0$ is trivial. We only need to show the compatibility with $[H_{i,1},H_{j,1}]=0$. Since we obtain
\begin{align*}
(\Psi^{m,m+n}_L\otimes\Psi^{m,m+n}_R)\circ\Delta^{m}=\Delta^+\circ\Psi^{m,m+n}.
\end{align*}
by a direct computation, we have
\begin{equation}
\Psi^{m,m+n}_L\otimes\Psi^{m,m+n}_R([\Delta^{m}(H_{i,1}),\Delta^{m}(H_{j,1})])=\Delta^+([\Psi^{m,m+n}(H_{i,1}),\Psi^{m,m+n}(H_{j,1})])=0.\label{ketu}
\end{equation}
By using \eqref{Eq2.11}-\eqref{Eq2.19}, we can write down $[\Delta^{m}(H_{i,1}),\Delta^{m}(H_{j,1})]$ as an element of the completion of $\bigotimes^2U(\widehat{\mathfrak{sl}}(m+n))$. By the PBW theorem of the affine Yangian, the relation \eqref{ketu} induces $[\Delta^{m}(H_{i,1}),\Delta^{m}(H_{j,1})]=0$.
\end{proof}

For a complex number $a\in\mathbb{C}$, we set a homomorphism called the shift operator of the affine Yangian:
\begin{equation*}
\tau_a\colon Y_{\hbar,\ve}(\widehat{\mathfrak{sl}}(n))\to Y_{\hbar,\ve}(\widehat{\mathfrak{sl}}(n))
\end{equation*}
determined by $X^\pm_{i,0}\mapsto X^\pm_{i,0}$ and $H_{i,1}\mapsto H_{i,1}+a H_{i,0}$. Then, by a direct computation, we obtain the compatibility with the coproduct for the affine Yangian and the parabolic presentation for a $W$-algebra.
\begin{Theorem}
We obtain the following relations:
\begin{gather*}
((\Phi_1\circ\tau_{-\gamma_{w}\hbar}\circ\Psi^{q_l,\min(q_1,q_w)}))\otimes\Phi_2)\circ\Delta^{q_l}=\Delta_W\circ\Phi\text{ if }q_1\geq q_l,\\
(\Phi_1\circ\tau_{-\gamma_{w}\hbar})\otimes(\Phi_2\circ\Psi^{q_1,\min(q_l,q_{w+1})})\circ\Delta^{q_1}=\Delta_W\circ\Phi\text{ if }q_1<q_l.
\end{gather*}
\end{Theorem}

\section*{Funding}
This work was supported by JSPS Overseas Research Fellowships, Grant Number JP2360303. 
\section*{Data Availability}
The authors confirm that the data supporting the findings of this study are available within the article and its supplementary materials.
\section*{Conflicts of interests/Competing interests}
The authors have no competing interests to declare that are relevant to the content of this article.

\end{document}